\documentclass[12pt,reqno]{amsart}
\usepackage{amssymb}
\usepackage{graphicx}
\usepackage{xcolor}
\usepackage{longtable}
\usepackage{float}

\usepackage[all]{xy}
\DeclareFontFamily{U}{mathb}{\hyphenchar\font45}
\DeclareFontShape{U}{mathb}{m}{n}{
      <5> <6> <7> <8> <9> <10> gen * mathb
      <10.95> mathb10 <12> <14.4> <17.28> <20.74> <24.88> mathb12
      }{}
\DeclareSymbolFont{mathb}{U}{mathb}{m}{n}
\DeclareMathSymbol{\righttoleftarrow}{3}{mathb}{"FD}

\theoremstyle{plain}
\newtheorem{prop}{Proposition}[section]
\newtheorem{theo}[prop]{Theorem}
\newtheorem{coro}[prop]{Corollary}
\newtheorem{lemm}[prop]{Lemma}
\theoremstyle{remark}

\theoremstyle{definition}
\newtheorem{rema}[prop]{Remark}

\newtheorem{exam}[prop]{Example}
\numberwithin{equation}{section}

\def\cO{{\mathcal O}}

\def\cT{{\mathcal T}}

\def\sA{{\mathsf A}}
\def\sD{{\mathsf D}}

\def\fA{{\mathfrak A}}

\def\fC{{\mathfrak C}}
\def\fD{{\mathfrak D}}

\def\fF{{\mathfrak F}}

\def\fS{{\mathfrak S}}

\def\fS{{\mathfrak S}}

\def\bA{{\mathbb A}}
\def\bG{{\mathbb G}}
\def\bP{{\mathbb P}}

\def\bZ{{\mathbb Z}}

\def\rH{{\mathrm H}}

\def\bF{{\mathbb F}}

\def\Pic{\mathrm{Pic}}

\def\Aut{\mathrm{Aut}}

\def\PSL{\mathsf{PSL}}
\def\GL{\mathsf{GL}}

\def\PGL{\mathsf{PGL}}

\def\Hom{\mathrm{Hom}}

\def\lim{\mathrm{lim}}

\def\Ker{\mathrm{Ker}}

\begin{document}
\title[Equivariant unirationality]{Equivariant unirationality of Fano threefolds}

\author[I. Cheltsov]{Ivan Cheltsov}
\address{Department of Mathematics, University of Edinburgh, UK}
\email{I.Cheltsov@ed.ac.uk}

\author[Yu. Tschinkel]{Yuri Tschinkel}
\address{
  Courant Institute,
  251 Mercer Street,
  New York, NY 10012, USA
}
\email{tschinkel@cims.nyu.edu}

\address{Simons Foundation\\
160 Fifth Avenue\\
New York, NY 10010\\
USA}

\author[Zh. Zhang]{Zhijia Zhang}

\address{
Courant Institute,
  251 Mercer Street,
  New York, NY 10012, USA
}

\email{zhijia.zhang@cims.nyu.edu}

\date{\today}

\begin{abstract}
We study unirationality of actions of finite groups on Fano threefolds with indices greater than one.
\end{abstract}

\maketitle

\section{Introduction}
\label{sec.intro}

In this paper, we consider generically free regular actions of finite groups $G$ on smooth projective varieties $X$ over an algebraically closed field  $k$ of characteristic zero. The simplest such actions arise from linear representations, i.e., as projectivizations $\bP(V)$ of linear representations $V$ of $G$. Note that not every faithful representation gives rise to a generically free action -- one also requires that the representation intersects the central $\bG_m$ trivially. 
We are interested in birational properties of $G$-actions, e.g., 
\begin{itemize}
\item {\bf (L)} {\em linearizability}: equivariant birationality to a linear action, 
\item {\bf (SL)} {\em stable linearizability}: linearizability of $X\times\bP(W)$, 
where $W$ is a $G$-representation; when the $G$-action on $X$ is generically free, we may take $W$ to be the trivial representation, by the no-name lemma. 
\end{itemize}
These notions should be viewed as equivariant analogs of rationality and stable rationality in geometry over nonclosed fields: there are many similarities but also striking differences between these theories, 
see, e.g., \cite{BnG}, \cite{KT-vector} for an introduction to this circle of problems.  

Here, we explore an equivariant version of unirationality: 
\begin{itemize}
    \item {\bf (U)} {\em $G$-unirationality}:
    existence of a dominant equivariant rational map
$$
\bP(V)\dashrightarrow X,
$$
where $V$ is a $G$-representation. 
\end{itemize}
Clearly, 
$$
\mathbf{(L)} \Rightarrow \mathbf{(SL)} \Rightarrow \mathbf{(U)}. 
$$
In the definition of {\bf (U)}, one can replace $\bP(V)$ by $V$, a $G$-representation, i.e.,  {\bf (U)} is equivalent to {\em very versality} of the $G$-action, in the terminology of, e.g., \cite{DR}. 

A necessary condition for $G$-unirationality is existence of fixed points upon restriction to abelian subgroups:

\begin{itemize}
    \item {\bf Condition (A)}:
    for every abelian subgroup $A\subseteq G$ one has 
    $$
    X^A\neq \emptyset.
    $$
\end{itemize}

This should be viewed as analogous to the existence of rational points, in the framework of birational geometry over nonclosed fields. Note that there do exist linearizable $G$-actions without $G$-fixed points, moreover, the existence of $G$-fixed points is not an equivariantly birational property, for nonabelian $G$. Thus, some standard unirationality constructions in geometry over nonclosed fields fail to apply in the equivariant setting.   

Duncan proved that Condition {\bf (A)} is also sufficient for $G$-unirationality of del Pezzo surfaces
of degree $\ge 3$, with generically free actions~\cite[Theorem 1.4]{Duncan}. The cases of del Pezzo surfaces of degree 2 and 1 remain open. 

We turn to smooth Fano threefolds. These are classified by 
\begin{itemize}
    \item $r\in\{1,\ldots,10\}$ -- rank of the Picard group,
    \item $i\in\{1,2,3,4\}$ -- Fano index, 
    \item $d= (-K_X)^3/i^3$ -- degree. 
\end{itemize}
Their rationality over algebraically closed fields is {\em almost} settled. There are still outstanding open problems concerning stable rationality and unirationality, e.g., unirationality of general quartic threefolds. Failure of stable rationality of very general cubic threefolds has been recently established  \cite{cubicengel}. Rationality properties of geometrically rational Fano threefolds have received a lot of attention, see \cite{HT-quad}, \cite{HT-18}, \cite{KP-1}, \cite{KP-2}. There is also a wealth of results concerning (stable) linearizability of group actions on (possibly singular) Fano threefolds, see, e.g., \cite{lemire}, 
\cite{CTZ-burk}, \cite{CTZ-cubic}, \cite{CMTZ}, \cite{HT-torsor}, \cite{BBT}, \cite{TT}. 

In this paper, we focus on equivariant unirationality of Fano threefolds of index $\ge 2$; these should be viewed as 3-dimensional analogs of del Pezzo surfaces. As in the work of Duncan on del Pezzo surfaces \cite{Duncan}, Fano threefolds of {\em small} degree are currently out of reach -- we do not even know the unirationality of any smooth Fano threefold of index $2$ and degree $1$ (all other Fano threefolds of index $\ge 2$ are known to be unirational, over algebraically closed fields). 

Our principal results are: 
\begin{itemize}
    \item A generically free $G$-action on a smooth quadric threefold is stably linearizable if and only if it satisfies Condition {\bf (A)}, Theorem~\ref{thm:quad-uni}. 
    \item A generically free $G$-action on a smooth cubic threefold satisfying Condition {\bf (A)} is $G$-unirational, with the possible exception of $G=C_9\rtimes C_3$, acting on the Fermat cubic threefold, 
  and $C_{11}\rtimes C_{5}, \PSL_2(\bF_{11})$, acting on the Klein cubic threefold, and a 1-parameter family of cubics with an $\fA_5$-action,    
    Theorem~\ref{thm:main-cubic}. 
    \item A generically free action on a cubic threefold with isolated singularities,  
    satisfying Condition {\bf (A)}, is $G$-unirational, Theorem~\ref{thm:singcub-main}.
    This class contains cubic threefolds with cohomological obstructions to stable linearizability, see \cite{CTZ-cubic}, \cite{CMTZ}. 
     \item A generically free $G$-action on a smooth  complete intersection of two quadrics $X\subset \bP^5$ is $G$-unirational if and only if it satisfies Condition {\bf (A)}, which turns out to be  equivalent to $X^G\neq \emptyset$, Theorem~\ref{thm:2-2}. 
     \item A generically free $G$-action on the (unique) smooth del Pezzo threefold $V_5$ of degree $5$ is linearizable except for $G=\fA_5$, which is not linearizable, by \cite{CS}, but stably linearizable, by \cite[Example 7]{BBT-der}, see Remark~\ref{exam:V5}.
\end{itemize}

\

Here is the roadmap of the paper: in Section~\ref{sect:gen} we recall notions of $G$-birational geometry, with a focus on $G$-unirationality. In Section~\ref{sect:cubic} we present several general unirationality constructions, applicable in the equivariant context. In Section~\ref{sect:quad} we specialize to quadric threefolds. Then we proceed to treat
smooth cubics in Section~\ref{sect:smooth-3}, singular cubics in Section~\ref{sect:sing}, and 
intersections of two quadrics in Section~\ref{sect:2quad}.

\

\noindent
{\bf Acknowledgments:} 
We are grateful to Andrew Kresch and Brendan Hassett for their interest and comments. We thank the referees for many helpful suggestions to improve the arguments in Sections~\ref{sect:quad} and~\ref{sect:smooth-3}.
The first author was supported by the Leverhulme Trust grant RPG-2021-229, by EPSRC grant EP/Y033485/1, and by Simons Collaboration grant \emph{Moduli of varieties}.
He also thanks CIRM, Luminy, for its hospitality and for
providing a perfect work environment. The second author was partially supported by NSF grant 2301983.

\section{Generalities}
\label{sect:gen}

Throughout, we work over an algebraically closed field $k$ of characteristic zero; all varieties are assumed to be irreducible over $k$. A $G$-variety is a projective variety with a generically free regular $G$-action; we will also consider situations when the $G$-action is {\em not} generically free. 
The $G$-fixed point locus of a $G$-variety $X$
will be denoted by $X^G$. For smooth projective $G$-varieties $X$ and {\em abelian} subgroups $A\subseteq G$, existence of $A$-fixed points is a birational property; this is not the case for nonabelian groups! Note that cyclic actions on rationally connected $G$-varieties always have fixed points, but this need not be the case for actions of abelian groups, e.g., the $C_2^2$-action on $\bP^1$ generated by $t\to -t$ and $t\to 1/t$ where $t$ is the affine coordinate on $\bP^1$. 

If $\tilde{X}\to X$ is a dominant morphism  of $G$-varieties and $\tilde{X}$ has $G$-fixed points, then so does $X$. Since {\em linear} actions of abelian groups always have fixed points, 
a necessary condition for $G$-unirationality of an action of a finite group $G$ on a smooth projective variety $X$ is Condition {\bf (A)}: existence of fixed points upon restriction to every abelian subgroup. 

Linearizability and stable linearizability of generically free actions of finite groups on rational surfaces and rational Fano threefolds have been considered in, e.g.,  \cite{HT-torsor}, \cite{BBT-der},  \cite{CTZ-burk}, \cite{CMTZ}, \cite{HT-quadric}. The related notion of $G$-unirationality, or, equivalently very versality, plays an important role in the study of {\em essential dimension}, see, e.g., \cite{DR}, \cite{Merkuriev}. However, only few concrete examples of $G$-unirational varieties failing (stable) linearizability are known, see \cite{Duncan} for a systematic treatment of del Pezzo surfaces of degree $\ge 3$.

The bridge between equivariant geometry and geometry over nonclosed fields is provided via {\em torsors}: For a field extension  
$K/k$ and a $G$-torsor $T$ over $K$, let 
${}^TX$ be the $T$-twist of $X$ over $K$, defined as the quotient ${}^TX:=(X_K\times T)/G$ by the diagonal $G$-action. A key result is:

\begin{theo} \cite[Theorem 1.1]{DR}
\label{thm:bridge}
Let $X$ be a $G$-variety over $k$. Then
\begin{itemize}
    \item the $G$-action on $X$ is stably linearizable if and only if for every field extension $K/k$ and every $G$-torsor $T$ over $K$, the twist ${}^TX$ is stably rational over $K$, 
    \item  the $G$-action on $X$ is unirational if and only if every twist ${}^TX$ as above is unirational over $K$. 
\end{itemize}
\end{theo}
We will apply this in subsequent sections to quadric and cubic hypersurfaces. 

\

A related notion that plays a role in equivariant geometry is 
the {\em Amitsur} group $\mathrm{Am}(X,G)$, defined in \cite[Section 6]{BC-finite} as the image of $G$-invariant divisor classes 
$$
\Pic(X)^G\stackrel{\delta_2}{\longrightarrow} \rH^2(G,k^\times), 
$$
where $\delta_2$ is the functorial homomorphism arising from the Leray spectral sequence, see, e.g., \cite[Section 3]{KT-dp}.
By \cite[Theorem 6.1]{BC-finite}, this is a $G$-birational invariant, which vanishes for linear actions and as soon  as $X^G\neq \emptyset$.
The vanishing of $\mathrm{Am}(X,G)$ implies that every $G$-invariant divisor class in $\Pic(X)$ is represented by a $G$-linearized line bundle on $X$. 

All possibilities for $\mathrm{Am}(X,G)$ for $G$-surfaces $X$ have been determined in \cite[Proposition 6.7]{BC-finite}. In particular, this group is trivial for del Pezzo surfaces $S$ of degree $\le 6$ when $\Pic(S)^G=\bZ$.

\begin{lemm}
\label{lem.Amitsurobstructs}
Let $X$ be a $G$-variety with nontrivial Amitsur invariant. Then 
$X$ is not $G$-unirational. 
\end{lemm}

\begin{proof}
Assume that $X$ is $G$-unirational. There exists a $G$-equivariant dominant rational map  $\pi:\bP(V)\dashrightarrow X.$ Resolving the indeterminancy by blowups, we have a $G$-equivariant morphism $\tilde \pi: \widetilde\bP{(V)}\to X$
which
gives rise to an inclusion
\[ \mathrm{Am}(X,G)\subseteq \mathrm{Am}(\widetilde\bP(V),G)=\mathrm{Am}(\bP(V),G). \]
This is clear from the functoriality of the map $\delta_2$ in the basic exact sequence (see \cite{STZ} or \cite[Section 6]{BC-finite})
$$
0\to \Hom(G,k^\times)\to \Pic(X,G)\to \Pic(X)^G\stackrel{\delta_2}{\longrightarrow} \rH^2(G,k^\times),
$$
where $\Pic(X,G)$ is the group of isomorphism classes of $G$-linearized line bundles on $X$,
see, e.g., \cite[Section 3]{KT-Brauer}. Since $\mathrm{Am}(\bP(V),G)=0$, we obtain a contradiction. 
\end{proof}

\begin{exam}
\label{exam:strange}
Note that the invariant $\mathrm{Am}(X,G)$ does not only depend on the effective action on $X$. It also depends on the generic stabilizer. Let $X=\bP^1$ with the action of the Klein four group $\bar{G}=C_2^2$; it is fixed point free. The Amitsur invariant $\mathrm{Am}(X,\bar{G})=\bZ/2$ is nontrivial, and the $\bar{G}$-action is not stably linearizable. However, $X$ admits a regular {\em non-generically free} action of the dihedral group $G=\mathfrak D_4$ (of order 8), with generic stabilizer $C_2$. 
As a variety with $G$-action (not necessarily generically free), $X$ is clearly $G$-unirational: we have a dominant $G$-equivariant map $V\dashrightarrow X$, where $V$ is an irreducible 2-dimensional representation of $G$. Of course,  every abelian subgroup of $G$ has fixed points on $X$. 
\end{exam}

Recall the definition of {\em Bogomolov multiplier}:
$$
\mathrm{B}_0(G):=\Ker\left(\rH^2(G,k^\times)\to \oplus_{A}\,  \rH^2(A,k^\times)\right), 
$$
where $A$ runs over all abelian subgroups of $G$. In fact, it suffices to consider {\em bicyclic} subgroups \cite{Bog-linear}, and to treat Sylow subgroups, one prime at a time. There is an extensive literature on this invariant, see, e.g., \cite[Section 3]{KT-Brauer}. 
In Example~\ref{exam:strange}, we have $\mathrm{B}_0(G)=0$. More generally, the Bogomolov multiplier vanishes when 
\begin{itemize}
\item $G$ is a $p$-group of order $|G|\le p^4$,
    \item $G$ is an extension of a cyclic group by an abelian group \cite[Lemma 4.9]{Bog-linear}, \cite[Lemma 3.1]{KT-Brauer}, 
    \item $G$ is a  split extension of a bicyclic group by an abelian group 
    \cite[Lemma 3.2]{KT-Brauer}.  
\end{itemize}

For a smooth projective $G$-variety $X$,  
Condition {\bf (A)} implies that 
$$
\mathrm{Am}(X,G)\subseteq \mathrm{B}_0(G). 
$$

\begin{lemm}
\label{lemm:pn}
Let $G$ be a finite group with $\mathrm{B}_0(G)=0$. Assume that $G$ acts regularly on $\bP^n$. This $G$-action is unirational if and only if it satisfies Condition {\bf (A)}.  
\end{lemm}

\begin{proof}
Under our assumption, one has that
$
\mathrm{Am}(\bP^n,G)\subseteq \mathrm{B}_0(G)=0.
$
This implies that $\cO_{\bP^n}(1)$ is $G$-linearized, and the corresponding homomorphism $G\to \PGL_n$ lifts to $G\to \GL_n$, i.e., the $G$-action on $\bP^n$ arises from a linear representation. The action is thus $G$-unirational by definition. 
\end{proof}

Here we do not assume that the $G$-action on $\bP^n$ is generically free. Note that for $G$ with $\mathrm{B}_0(G)\neq 0$, there exist regular generically free actions on projective spaces, which satisfy Condition {\bf (A)} but are {\em not} linear, and not $G$-unirational.

\begin{lemm}
\label{lemm:line}
Let 
$$
X:=\bP_B(\mathcal E) 
$$
be
the projectivization of a vector bundle $\mathcal E\to B$, where $B$ is a smooth projective $G$-variety. 
Assume that $X$ carries a generically free regular $G$-action such that the canonical projection 
$$
\pi: X\to B
$$
is $G$-equivariant. Assume that 
\begin{itemize}
\item[(1)] $\rH^1(G,\Pic(B))=0$, 
\item[(2)] $X$ satisfies Condition {\bf (A)}, and 
\item[(3)] $\mathrm{B}_0(G)=0$.
\end{itemize}
Then the $G$-action lifts to $\mathcal E$. 
\end{lemm}

\begin{proof}
Consider the sequence of $G$-modules
$$
0\to \Pic(B)\stackrel{\pi^*}{\longrightarrow} \Pic(X)\to \bZ\to 1. 
$$
Note that we do not assume that the $G$-action on $B$ is generically free.

From the long exact sequence
$$
0\to \Pic(B)^G\stackrel{\pi^*}{\longrightarrow} \Pic(X)^G\to \bZ\to \rH^1(G, \Pic(B))
$$
and assumption (1), we obtain that there is a $G$-invariant class in $\Pic(X)$ surjecting onto the class of the relative $\mathcal O(1)$.  By assumption (2), every line bundle on $X$ admits a linearization upon restriction to abelian subgroups, i.e., $\mathrm{Am}(X,G)\subseteq \mathrm{B}_0(G)$. By assumption (3), that latter group is trivial, which implies that 
every invariant line bundle on $X$ is $G$-linearizable, in particular, the $G$-action lifts to $\mathcal E$.  
\end{proof}

\section{Unirationality constructions}
\label{sect:cubic}

General results linking $G$-unirationality of actions on quadric and cubic hypersurfaces to geometry over nonclosed fields are in \cite[Section 10]{DR}. We complement these considerations by providing explicit equivariant unirationality constructions in many cases; this allows us to essentially settle unirationality of $G$-actions on quadric and cubic threefolds, in subsequent sections.

We start with obvious observations: 
\begin{itemize}
    \item $G$-unirationality is compatible with taking products; 
    \item existence of a dominant $G$-equivariant rational map 
    $X\dashrightarrow Y$ from a 
    $G$-unirational $X$ implies $G$-unirationality of $Y$. 
\end{itemize}

The following is an analog of a standard result in the theory of quadrics over nonclosed fields, see, e.g., \cite[Chapter 4]{EKMbook}.

\begin{prop}
\label{prop:quad-genu}
Let $G$ be a finite group and $V$ a representation of $G$ giving rise to a generically free $G$-action on a smooth quadric hypersurface $X\subset \bP(V)$, of dimension $\ge 2$.    
Assume that there is a $G$-invariant irreducible subvariety $S\subset X$,
which is $G$-unirational. Then $X$ is $G$-unirational. 
\end{prop}

Here, we allow $S$ to have a nontrivial generic stabilizer, i.e., we require that there is a dominant rational map $\bP(W)\dashrightarrow S$, for some $G$-representation $W$.  

\begin{proof}
By assumption, there is a dominant $G$-equivariant rational map $\bP(W)\dashrightarrow S$. 
Consider the rational map
$$
S\times \bP(V) \dashrightarrow  X,
$$
sending a pair of points $(s,p)\in S\times \bP(V)$ to the second 
intersection point of $X$ with the line 
$\mathfrak l(s,p)$ through $s$ and $p$. This is clearly dominant and $G$-equivariant, ensuring the $G$-unirationality of $X$.  
\end{proof}

We continue with general $G$-unirationality constructions for cubic hypersurfaces of dimensions $\ge 2$. Over fields $K$ of characteristic zero unirationality of smooth cubic hypersurfaces with $X(K)\neq \emptyset$ is classical; generalizations to singular cubic hypersurfaces that are not cones can be found in \cite{kollarcubic}.

\begin{prop}
    \label{prop:ind-1}
Let $V$ be a representation of a finite group $G$. 
    Let $X\subset \bP(V)$ be a $G$-invariant irreducible cubic hypersurface of dimension $\geq 3$ which is not a cone. Assume that 
    the $G$-action on $X$ is generically free and that $X^G\neq \emptyset$. Then $X$ is $G$-unirational.
\end{prop}

\begin{proof}
This is essentially \cite[Corollary 10.6]{DR}, extended to singular cubic hypersurfaces. In particular, if $G$ fixes a singular point, then the $G$-action is linearizable. If $G$ fixes a smooth point, then the same argument in \cite[Corollary 10.6]{DR} applies.
\end{proof}

\begin{prop}
    \label{prop:ind-2}
Let $V$ be a representation of a finite group $G$. 
    Let $X\subset \bP(V)$ be a $G$-invariant cubic hypersurface which is not a cone. Assume that 
    the $G$-action on $X$ is generically free 
   and that $G$ has normal subgroup $H$ such that $X$ is $H$-unirational and $G/H$ is a 2-group. Then $X$ is $G$-unirational.
\end{prop}

\begin{proof}
Since $G/H$ is a 2-group, we know there exists a tower 
$$
H\subset G_1\subset G_2\subset\cdots G
$$
where each inclusion has index 2.  By induction, it suffices to consider when $|G/H|=2$. In this case, the assertion essentially follows from \cite[Theorem 3.2]{Duncan}, see also \cite[Remark 3.3]{Duncan}: $G$-unirationality is equivalent to $K$-unirationality of all twists of $X$ over $K$ (via $G$-torsors), and for cubic hypersurfaces, unirationality over quadratic extensions of $K$ is equivalent to unirationality over $K$.     
\end{proof}

\begin{coro}\label{coro:2sing}
Let $V$ be a representation of a finite group $G$. 
Let $X\subset \bP(V)$ be a $G$-invariant cubic hypersurface which is not a cone. Assume that 
    the $G$-action on $X$ is generically free 
   and there exists a $G$-orbit of points on $X$ of length one or two. 
Then the $G$-action on $X$ is $G$-unirational.    
\end{coro}

\begin{proof}
First we consider when  $X$ has a $G$-fixed point $\mathfrak p$. It is $G$-unirational by Proposition~\ref{prop:ind-1} if $\mathfrak p$ is smooth. If $\mathfrak p$ is singular, $X$ is $G$-linearizable, and thus $G$-unirational by projection from $\mathfrak p$. When  $X$ has a $G$-orbit $\{\mathfrak p_1,\mathfrak p_2\}$ of points of length 2, there exists an index-2 subgroup $H\subset G$ fixing $\mathfrak p_1$ and $\mathfrak p_2$. By the argument in the previous case, we know that $X$ is $H$-unirational. The $G$-unirationality of $X$ then follows from Proposition~\ref{prop:ind-2}.
\end{proof}

\begin{prop}
\label{prop:cubic-uni-2}
Let $V$ be a faithful $n$-dimensional representation of a finite group $G$ with $n\geq 5$. 
    Let $X\subset \bP(V)$ be a $G$-invariant cubic hypersurface which is not a cone.
    Assume that 
    the $G$-action on $X$ is generically free  and that there is a $G$-invariant subvariety $S\subset X$, which is $G$-unirational. Then $X$ is $G$-unirational. 
\end{prop}

Here we do not require that the $G$-action on $S$ is generically free. 

\begin{proof}
For any field extension $K/k$ and every $G$-torsor $T$ over $K$, the twist $^TS$ is a $K$-unirational subvariety of $^TX$ by Theorem~\ref{thm:bridge}, which implies that $^TX$ has a $K$-point. By assumptions, $^TX$ is a cubic hypersurface in $^T\bP(V)=\bP^n$. It follows from \cite{kollarcubic} that $^TX$ is $K$-unirational, and thus $X$ is $G$-unirational by Theorem~\ref{thm:bridge} again. 

\end{proof}

\begin{rema}\label{rema:explicitcub}
    One can also find explicit $G$-unirational constructions in many cases under the assumptions of Proposition~\ref{prop:cubic-uni-2}. For example, assume that $S$ is an irreducible hyperplane section, i.e., a cubic surface. Let $\cT\to X^\circ$ be the tangent bundle over the smooth locus $X^\circ$ of $X,$ and $\cT\vert_{S^\circ}$ the restriction of $\cT$ to the smooth locus of $S$. Then $\cT\vert_{S^\circ}$ is a rank-$(n-2)$ vector bundle and the $G$-action lifts to $\cT\vert_{S^\circ}$. By \cite[Theorem 4.2]{starsmooth} and \cite[Theorem 2.5]{starsing}, we know that the intersection of the tangent space at the generic point of $S^\circ$  with $X$ is not a cone. Thus, each point in the fiber above a general point $\mathfrak p\in S^\circ$ canonically corresponds to a line in $\bP(V)$ that is tangent to $X$ at $\mathfrak p$ with multiplicity $2$, and thus intersects $X$ at a unique point $\mathfrak q\neq \mathfrak p$. Sending each point in $\cT\vert_{S^\circ}$ to the corresponding point $\mathfrak q\in X$, we obtain a $G$-equivariant rational map
\begin{align}\label{eqn:tangentdominant}
    \varphi:\cT\vert_{S^\circ}\dashrightarrow X.
\end{align}
We show  that $\varphi$ is dominant. Assume that
$$
X=\{f_3(x_1,\ldots,x_{n})=0\} \quad \text{ and }\quad S=X\cap\{x_1=0\}.
$$ 
Given a general point $\mathfrak q=[q_1:\cdots:q_{n}]\in X$, we can find general $p_1,\ldots,p_{n-1}\in k$ such that 
$$
    \sum_{i=1}^{n}\frac{\partial f_3}{\partial x_i}(0,p_1,\ldots,p_{n-1}) q_i=f_3(0,p_1,\ldots,p_{n-1})=0.
$$
Then the line passing through $\mathfrak q$ and $\mathfrak p=[0:p_1:\ldots:p_{n-1}]$ is tangent to $X$ at $\mathfrak p$. It follows that \eqref{eqn:tangentdominant} is dominant. 

It remains to show that $\cT\vert_{S^\circ}$ is $G$-unirational. We apply the no-name lemma to the vector bundle 
$$
\cT|_{S^\circ} \times V \to S^\circ\times V,
$$
which  yields a $G$-equivariant birationality 
$$
\cT|_{S^\circ} \times V \sim_G S\times V \times\bA^{n-2},
$$
with trivial action on $\bA^{n-2}$. By our assumptions,  $\cT\vert_{S^\circ}\times V$ is $G$-unirational, implying the $G$-unirationality of $\cT\vert_{S^\circ}$.
\end{rema}

\begin{rema}
\label{rema:canbe}
When $S$ is a $G$-invariant linear subspace, e.g., a line, contained in $X$, the assumption that $S$ is $G$-unirational follows from the assumption that $G$-acts on the ambient projective space via a linear representation $V$. Thus, any cubic hypersurface of dimension $\geq 3$ with a linear $G$-action containing a $G$-invariant line or a $G$-invariant plane is $G$-unirational by Proposition~\ref{prop:cubic-uni-2}.
\end{rema}

\section{Quadrics}
\label{sect:quad}

Stable linearizability of certain $G$-quadric threefolds $X\subset \bP^4$ has been considered in \cite[Sections 5 and 6]{HT-quadric}: 
all but one subgroup $G$ of the Weyl group $W(\mathsf D_5)$, acting via signed permutations on the diagonal quadric 
$$
\sum_{i=1}^5 x_i^2=0, 
$$
are stably linearizable, provided Condition {\bf (A)} is satisfied. The only open case was $G=\mathfrak D_4$, acting on
$$
\bP(\chi_1\oplus \chi_2\oplus \chi_3\oplus V_2), 
$$
where $V_2$ is the unique faithful 2-dimensional representation, and $\chi_1,\chi_2,\chi_3$ are distinct characters of $G$. There are also many actions that do not arise from subgroups of 
$W(\mathsf D_5)$, e.g., an action of $\mathfrak D_8$ on the quadric 
$$
x_1^2+x_2x_3+x_4x_5=0, 
$$
arising from an action on 
$$
\bP(\mathbf 1 \oplus V_2\oplus V_2'),
$$
where $V_2$ is a faithful representation and $V_2'$ a nonfaithful 2-dimensional representation, see \cite[Remark 6.6]{HT-quadric}. 
The following theorem covers these and other actions, settling the stable linearizability problem for actions of finite groups on smooth quadric threefolds. The key observation is that, when Condition {\bf (A)} is satisfied, there always exists a $G$-invariant unirational quadric surface or conic in $X$ and Proposition~\ref{prop:quad-genu} applies.

\begin{theo} 
    \label{thm:quad-uni}
    Let $X\subset \bP^4$ be a smooth quadric threefold, with a generically free action of a finite group $G$. Then the $G$-action is stably linearizable if and only if it satisfies Condition {\bf (A)}.   
\end{theo}

The rest of this section is devoted to the proof of this theorem.

By Theorem~\ref{thm:bridge}, stable linearizability of the $G$-action is equivalent to stable rationality of every twist ${}^TX$ over every field extension $K/k$. For quadrics, (stable) rationality and unirationality are equivalent over any field. Again by Theorem~\ref{thm:bridge}, this is equivalent to $G$-unirationality of the action. 
By \cite[Theorem 10.2]{DR}, $G$-unirationality  is equivalent to $G_2$-unirationality, where $G_2$ is the 2-Sylow subgroup of $G$. We may also assume that $G$ is nonabelian: abelian group actions on $X$ satisfying Condition {\bf(A)} are linearizable via the projection from a fixed point.

Since $X\subset\bP^4$ is a $G$-invariant quadric, we know that $\cO_X(1)$ is $G$-linearized. It follows that the $G$-action on the ambient $\bP^4$ arises from the projectivization of a linear representation of $G$. Thus, to show Theorem~\ref{thm:quad-uni}, we now assume that $G$ is a nonabelian 2-group and $X\subset \bP(V)$, where $V$ is a faithful representation of $G$.  Since 
$G$ is a 2-group, we have the following cases: 
\begin{itemize}
    \item[(1)] $V=\chi_1\oplus\chi_2\oplus \chi_3\oplus V_2$, where $V_2$ is an irreducible $G$-representation, and $\chi_j$ are characters of $G$ for $j=1,2,3$, 
    \item[(2)]  $V=\chi_1\oplus V_2 \oplus V_2'$, where $V_2$ and $V_2'$ are irreducible $G$-representations, 
     \item[(3)]
     $V=\chi_1\oplus V_4$, where $V_4$ is an irreducible representation of $G$.
\end{itemize}
We proceed with a case-by-case analysis. 

\ 

{\bf\em Case} (1): Up to isomorphism, we may assume that $X$ is given by 
$$
X=\{x_1^2+x_2^2+x_3^2+x_4x_5=0\}\subset\bP^4,
$$ 
with 
$G$ acting via characters on $x_1,x_2, x_3$, and via the irreducible representation $V_2$ on $x_4,x_5$. 
We claim that Condition {\bf (A)} for $X$ implies Condition {\bf (A)} for the $G$-action on the conic $C\subset \bP^2=\bP(\chi_1\oplus \chi_2\oplus\chi_3)$, given by $x_4=x_5=0$. 
Note that the $G$-action on this $\bP^2$ has a nontrivial generic stabilizer, since $G$ is nonabelian, and the effective action on the conic $C$ may not satisfy Condition {\bf (A)}.


Since $V_2$ is irreducible, there exist elements of the form 
$$
\mathrm{diag}(\pm1,\pm1,\pm1,a,a^{-1})\in G,
$$
where $a$ is a 2-power root of unity of order $\mathrm{ord}(a)\geq 4$. Then 
$$
\varepsilon_1=\mathrm{diag}(1,1,1,-1,-1)\in G.
$$
Note that 
$$
X^{\langle\varepsilon_1\rangle}=C\cup \mathfrak p_1\cup \mathfrak p_2,
$$
where 
$$
\mathfrak p_4=[0:0:0:1:0],\quad \mathfrak p_5=[0:0:0:0:1].
$$

Let $A\subset G$ be an abelian subgroup containing $\varepsilon_1$, 
fixing points on $X$ but not on $C$. Then $A$ fixes $\mathfrak p_4, \mathfrak p_5$, and  acts via $C_2^2$ on $C$. It follows that $A$ contains 
$$
\varepsilon_2=\mathrm{diag}(-1,1,1,a_1,a_1^{-1}),\quad\varepsilon_3=\mathrm{diag}(1,-1,1,b_1,b_1^{-1}),
$$
for some 2-power roots of unity $a_1, b_1$. Depending on $a_1$ and $b_1$, the elements $\varepsilon_2$ and $\varepsilon_3$ generate at least one of the following elements
$$
\mathrm{diag}(-1,1,1,-1,-1),\quad \mathrm{diag}(1,-1,1,-1,-1),\quad \mathrm{diag}(1,1,-1,-1,-1).
$$
Without loss of generality, assume that $A$ contains 
$$
\varepsilon_4=\mathrm{diag}(-1,1,1,-1,-1).
$$
By the irreducibility of $V_2$,  $G$ also contains an element switching $\mathfrak p_4$ and $\mathfrak p_5$. Up to multiplying by $\varepsilon_2$ and $\varepsilon_3$, we may assume that $G$ contains 
$$
\sigma:(x_1,\ldots,x_5)\mapsto(x_1,-x_2, x_3,a_2x_5,a_2^{-1}x_4),
$$
for some $a_2$. The subgroup 
$$
\langle\varepsilon_1,\varepsilon_4,\sigma\rangle\simeq C_2^3
$$
does not fix points on $X$.

Let $A\subset G$ be an abelian subgroup not containing $\varepsilon_1$, fixing points on $X$ but not on $C$. Then $\langle A,\varepsilon_1\rangle$ is an abelian group whose fixed locus is $\{\mathfrak p_4,\mathfrak p_5\}$ or empty. In the first case, we repeat the argument above. We conclude that Condition {\bf (A)} for the $G$-action  on $X$ implies 
Condition {\bf (A)} for the (not generically free) $G$-action on $C$.

Next, 
we observe that 
$G$ preserves the set $\{\mathfrak p_4,\mathfrak p_5\}$, which yields an exact sequence 
$$
0\to N\to G\to C_2\to 0,
$$
where $N$ is the maximal subgroup fixing $\mathfrak p_4$ and $\mathfrak p_5$, acting diagonally on $x_1,\ldots, x_5$. Thus $N$ is abelian. By \cite[Lemma 3.1]{KT-Brauer},  $\mathrm B_0(G)=0$, since $G$ is an extension of a cyclic group by an abelian group.  By Lemma~\ref{lemm:pn}, $C$ is $G$-unirational since Condition {\bf(A)} is satisfied on $C$. By Proposition~\ref{prop:quad-genu}, $X$ is $G$-unirational.

 \

{\bf\em Case} (2): We may assume that $G$ acts on 
$$
X=\{x_1^2+x_2x_3+x_4x_5=0\},
$$
via $\chi_1$ on $x_1$, via $V_2$ on $x_2,x_3$, and via $V_2'$ on $x_4,x_5$. By the irreducibility of $V_2$ and $V_2'$, $G$ acts on $\bP(V_2)$ and $\bP(V_2')$ via dihedral groups. 
Let
$$
\mathfrak p_2=[0:1:0:0:0],\quad \mathfrak p_3=[0:0:1:0:0],
$$
$$
\mathfrak p_4=[0:0:0:1:0],\quad \mathfrak p_5=[0:0:0:0:1].
$$
Since $G$ leaves invariant the sets $\{\mathfrak p_2,\mathfrak p_3\},$ and $\{\mathfrak p_4,\mathfrak p_5\}$ respectively, there is an exact sequence
\begin{align}\label{eqn:exactseq1+2+2}
    0\to N\to G\to Q\to 0,
\end{align}
where $Q=C_2$ or $C_2^2$, depending on the induced action on $\{ \mathfrak p_2,\mathfrak p_3,\mathfrak p_4,\mathfrak p_5\}$, and $N\subset G$ is the maximal subgroup fixing these points; $N$ consists of elements 
\begin{align}\label{eqn:var1+2+2}
    \varepsilon_1=\mathrm{diag}(1,t_1,t_1^{-1},t_2,t_2^{-1}),\quad t_1, t_2\text{ are $2$-power roots of unity.}
\end{align}
The (nontrivial) center of $G$ acts via scalars on $V_2,V_2'$, i.e., it must contain 
one of the following 
$$
\varepsilon_2=\mathrm{diag}(1,-1,-1,1,1), \quad \varepsilon_3=\mathrm{diag}(1,1,1,-1,-1),\quad \varepsilon_2\varepsilon_3.
$$
If $Q=C_2^2=\langle(1,2),(3,4)\rangle$, we may assume that $G$ also contains 
$$
m_1: (x_1,\ldots,x_5)\mapsto(x_1,a_1x_3,a_1^{-1}x_2,a_2x_4,a_2^{-1}x_5),
$$
$$
m_2: (x_1,\ldots,x_5)\mapsto(x_1,b_1x_2,b_1^{-1}x_3,b_2x_5,b_2^{-1}x_4),
$$
where $m_1$ is a lift of $(1,2)$ and $m_2$ is a lift of $(3,4)$.
After a change of variables, we have $a_1=b_2=1$, and $a_2$ and $b_1$ are 
2-power roots of unity: 
$$
m_1^2=\mathrm{diag}(1,1,1,a_2^2,1/a_2^2),
\quad
m_2^2=\mathrm{diag}(1,b_1^2,1/b_1^2,1,1).
$$
If $a_2^2\ne1$ and $b_1^2\ne1$, then $G$ contains $\varepsilon_2$ and $\varepsilon_3$. The abelian group 
\begin{equation} 
\label{eqn:c23}
\langle\varepsilon_2,\varepsilon_3,m_1m_2\rangle\simeq C_2^3
\end{equation}
does not fix points on $X$. If $a_2^2=1$ and $b_1^2\ne1$, then $\varepsilon_2\in G,$ and $N$ contains an element $\varepsilon_1$ as in \eqref{eqn:var1+2+2} with $t_2\ne\pm1$ since otherwise the representation  on $x_4,x_5$ is reducible. If $\mathrm{ord}(b_1)\ge \mathrm{ord}(t_1)$, or $\mathrm{ord}(t_2)\geq \mathrm{ord}(t_1)$, then $\varepsilon_3$ is generated by $m_2^2$ and $\varepsilon_1$, and $G$ contains the $C_2^3$ from \eqref{eqn:c23} with no fixed points. If $\mathrm{ord}(t_1)>\mathrm{ord}(b_1), \mathrm{ord}(t_2)$, one can find $r_1,r_2,r_3\in \bZ$ such that
\begin{align*}
   m_1^{r_1}\varepsilon_1^{r_2} &:(x_1,\ldots,x_5)\mapsto(x_1,c_1x_3,c_1^{-1}x_2,-x_4,-x_5),\\
        m_2\varepsilon_1^{r_3} &:(x_1,\ldots,x_5)\mapsto(x_1,x_2,x_3,c_2x_5,c_2^{-1}x_4)
\end{align*}
for some $c_1, c_2$. Then 
$$
\langle m_1^{r_1}\varepsilon_1^{r_2},
 m_2\varepsilon_1^{r_3},
\varepsilon_2\rangle\simeq C_2^3
$$
does not fix points on $X$. By symmetry,  the same applies when $b_1^2=1, a_2^2\ne1$.

We are reduced to the case $a_2^2=b_1^2=1$, with  $\langle m_1,m_2\rangle\simeq C_2^2$. By the irreducibility of $V_2$, $N$ contains an $\varepsilon_1$ as in \eqref{eqn:var1+2+2} of order at least $4$. We repeat the argument above to find an abelian group failing Condition {\bf (A)},  after multiplying $m_1$ and $m_2$ by $\varepsilon_1$.

 Thus, if the $G$-action on $X$ satisfies Condition {\bf(A)}, we have
$$
Q=C_2=\langle(1,2)(3,4)\rangle
$$
and $G$ contains a lift of $(1,2)(3,4)$ of the form 
$$
m_3: (x_1,\ldots,x_5)\mapsto(x_1,d_1x_3,d_1^{-1}x_2,d_2x_5,d_2^{-1}x_4)
$$
for some $d_1,d_2$. We discuss cases based on the center of $G$:
\begin{itemize}
\item 
If $\varepsilon_2, \varepsilon_3\in G$, then 
$$
\langle \varepsilon_2, \varepsilon_3,m_3\rangle\simeq C_2^3
$$
has no fixed points on $X$.
    \item 
    If $\varepsilon_2\varepsilon_3\in G$, then every abelian subgroup of $G$ fixes a point on the quadric $S=X\cap\{x_1=0\}$.  Otherwise, we can find an abelian group without fixed points by adjoining $\varepsilon_2\varepsilon_3$.
\item 
If $\varepsilon_2\in G$ and $\varepsilon_3\not \in G$, then every abelian subgroup $A\subset G$ fixes a point on $C=X\cap\{x_2=x_3=0\}$. Otherwise, $A'=\langle A,\varepsilon_2\rangle$ fixes $\{\mathfrak p_2,\mathfrak p_3\}$, and does not surject onto $Q$, which implies that it also fixes $\mathfrak p_4,\mathfrak p_5\in C$, contradiction. 
\item 
If $\varepsilon_3\in G$ and $\varepsilon_2\not \in G$, the same argument shows that every abelian subgroup of $G$ fixes a point on $C=X\cap\{x_4=x_5=0\}$.
\end{itemize}

Note that we have $\mathrm{B_0}(G)=0$, since $G$ is an extension of $C_2$ by an abelian group. In the last two cases, the corresponding conic is $G$-unirational by Lemma~\ref{lemm:pn}, and $X$ is $G$-unirational by Proposition~\ref{prop:quad-genu}.  In the second case, the $G$-action on the quadric $S=\bP^1\times\bP^1$ does not switch the two rulings.  By Lemma~\ref{lemm:pn},  we can lift the $G$-action on each $\bP^1$ factor to $\bA^2$, i.e., the $G$-action on $S$ to $\bA^2\times\bA^2=\bA^4$. It follows that
$S$ is $G$-unirational. By Proposition~\ref{prop:quad-genu}, $X$ is also $G$-unirational. 

\

{\bf\em Case} (3): We show that Condition {\bf (A)} is never satisfied. Since $V_4$ is irreducible, 
there are no fixed points in $\bP(V_4)$, and $S:=\bP(V_4)\cap X$ is smooth. 
We may assume that 
    $$
    X=\{x_1^2+x_2x_5-x_3x_4=0\}\subset\bP^4=\bP(\mathbf 1\oplus V_4),
    $$
    with $G$ acting trivially on $x_1$, and via $V_4$ on $x_2,x_3,x_4,x_5$. Since $V_4$ is faithful and $G$ preserves $X$, the center of $G$ is $C_2$, generated by
    $$
\varepsilon_1=\mathrm{diag}(1,-1,-1,-1,-1).  
    $$
   Finite 2-subgroups of $\Aut(S)=\Aut(\bP^1\times\bP^1)$ are subgroups of $\fD_n\wr C_2$,  for $n$ some power of $2$. This yields 4 distinguished points  
$$
\mathfrak p_2=[0:1:0:0:0],\quad \mathfrak p_3=[0:0:1:0:0],
$$
$$
\mathfrak p_4=[0:0:0:1:0],\quad \mathfrak p_5=[0:0:0:0:1],
$$   
and an exact sequence 
$$
0\to N\to G\to Q\to 0,
$$
where $N$ fixes each $\mathfrak p_i$. Since $V_4$ is irreducible, 
the image $Q\subset \fS_4$, acting via permutations on $\{\mathfrak p_2,\mathfrak p_3,\mathfrak p_4,\mathfrak p_5\}$, can be $C_2^2, C_4$, or $\fD_4$. 
The subgroup $N$ is generated by elements of the form 
$$
\mathrm{diag}(1,t_2,t_1,t_1^{-1},t_2^{-1}), \quad t_1,t_2\text{ are }2\text{-power roots of unity.}
$$
We proceed to analyze each case. 

\ 

$\bullet$ $Q=C_2^2$, with lifts of $(1,2)(3,4)$ and $(1,3)(2,4)$:
\begin{align}\label{eqn:m1m2}
    m_1&:(x_1,\ldots,x_5)\mapsto(x_1,a_2x_3,a_1x_2,-a_1^{-1}x_5,-a_2^{-1}x_4),\\
    m_2&:(x_1,\ldots,x_5)\mapsto(x_1,b_2x_4,b_1x_5,-b_1^{-1}x_2,-b_2^{-1}x_3),\notag
\end{align}
for some $a_1,a_2,b_1,b_2$. Put $s_1=a_1a_2, s_2=b_2/b_1$ so  that 
\begin{align*}
    m_1^2&=\mathrm{diag}(1,s_1,s_1,1/s_1,1/s_1),\\
m_2^2&=\mathrm{diag}(1,-s_2,-1/s_2,-s_2,-1/s_2).\notag
\end{align*}
Since $G$ is a 2-group, $s_1$ and $s_2$ are 2-power roots of unity. 
When the orders $\mathrm{ord}(s_1), \mathrm{ord}(s_2)\ge 4$, we can find integers $r_1,r_2$ such that 
$$
s_2^{r_2}=\zeta_4,\quad s_1^{r_1}=\zeta_4^3.
$$
Setting
$$
\varepsilon_3:=\varepsilon_1^{r_2}m_1^{2r_1} m_2^{2r_2}=\mathrm{diag}(1,1,-1,-1,1),
$$
we have $S^{\langle\varepsilon_1, \varepsilon_3\rangle}=\{ \mathfrak p_2,\ldots,\mathfrak p_5\}$. The abelian group 
\begin{align}\label{eqn:V4C2^3}
\langle \varepsilon_3,m_1m_2, \varepsilon_1\rangle\simeq C_2^3
 \end{align}
 has no fixed points on $X$. Thus, $s_1$ or $s_2$ equals $\pm 1$. By symmetry, we may assume that $s_2=\pm 1$. Then the line  
$$
\mathfrak l:=\{x_1=x_2+x_3\sqrt{a_2/a_1}=x_4\pm x_5\sqrt{a_2/a_1}=0\} \subset \bP(V)
$$
is $\langle m_1,m_2\rangle$-invariant. 
By the irreducibility of $V_4$, 
the kernel $N$ contains an element $\varepsilon_2$ not leaving $\mathfrak l$ invariant. In particular, $\varepsilon_2$ takes the form
\begin{equation} 
\label{eqn:varep2}
\!\!\!\varepsilon_2=\mathrm{diag}(1,t_2,t_1,t_1^{-1},t_2^{-1})\neq\varepsilon_1,
\end{equation}
for some 2-power roots of unity $t_1,t_2$ such that $t_1\ne t_2$. 
 If $\mathrm{ord}(t_2/t_1)\ge 4$, we repeat the argument above, replacing $m_1$ by $\varepsilon_2m_1$ to find a $C_2^3$ with no fixed points on $X$. 
Thus we are reduced to  $t_1=-t_2$. 

Similarly, if $s_1=\pm1$, we are reduced to $t_1=-1/t_2$. In particular, for 
$s_1,s_2\in\{\pm1\}$, we are reduced to $t_1,t_2\in\{\pm1\}$. In this case, $G$ contains $\varepsilon_3$ and a $C_2^3$ with no fixed points as in \eqref{eqn:V4C2^3}.

We consider the case when $s_1\ne \pm1$. As is discussed, we may assume $t_1=-t_2$.  
If $\mathrm{ord}(t_1)\le \mathrm{ord}(s_1)$, then $\varepsilon_3\in \langle \varepsilon_2, m_1^2\rangle\subset G$, and   
$G$ contains a $C_2^3$ with no fixed points. 
In the opposite case, there is a $d_1\in\bZ$ such that 
$$
t_1^{2d_1}=(-1)^{d_1}/s_1.
$$
Setting $d_2=(3-s_2)/2$, we find that 
$$
\langle  m_1\varepsilon_2^{d_1}, m_2\varepsilon_2^{d_2}, \varepsilon_1\rangle\simeq C_2^3
$$
has no fixed points on $X$.

\

$\bullet$  $Q=\fD_4$, with $G$ containing $m_1,m_2$ as in \eqref{eqn:m1m2}. 
If $G$ contains 
$$
\varepsilon_3=\mathrm{diag}(1,1,-1,-1,1),
$$
there is a $C_2^3$ with no fixed points on $X$.  To show this, 
we are reduced to the case when $s_1=\pm1$ or $s_2=\pm1$, as above.
Consider the lift  of $(1,3,4,2)\in\fS_4$:
\begin{align}\label{eqn:m3D4quadric}
m_3:(x_1,\ldots,x_5)\mapsto (x_1,c_2x_4,c_1x_2,-c_1^{-1}x_5,-c_2^{-1}x_3),
\end{align}
for some $c_1,c_2$. Observe that
\begin{align}\label{eqn:m3}
   (m_2m_3)^2&=\mathrm{diag}(1,c_2^2/b_1^2,1,1,b_1^2/c_2^2), \\
   (m_1m_3)^2&=\mathrm{diag}(1,1,a_2^2c_1^2,1/(a_2^2c_1^2),1).\notag
\end{align}
We may assume that $c_1=\pm 1/a_2$ and $c_2=\pm b_1$, otherwise $\varepsilon_3$ can be generated by elements in \eqref{eqn:m3}. Assume $c_1=1/a_2$ and $c_2=b_1$ (the other cases are similar). If $s_1\ne -s_2$, $\varepsilon_3$ is some power of $m_1m_2m_3^2$. When $s_1=-s_2=\pm 1$, we find an $\langle m_1,m_2,m_3\rangle$-invariant line:
$$
\mathfrak l=\{x_1=x_2+\frac{b_1}{a_1}x_5=x_3\mp b_1a_1x_4=0\} \subset V_4.
$$
Thus $N$ contains an $\varepsilon_2$ as in \eqref{eqn:varep2} which does not leave $\mathfrak l$ invariant. This implies that $t_1^2, t_2^2\ne1$. As above, when $t_1\ne\pm t_2$ or $t_1t_2\ne \pm1$, we have $\varepsilon_3\in G$. When $t_1=\pm t_2$ and $t_1t_2=\pm1$, we know $t_1^2=t_2^2=-1$. Then $
\varepsilon_3=(m_1m_3\varepsilon_2)^2\in G. 
$

\

$\bullet$  $Q=C_4$, with $G$ containing a lift $m_3$ of $(1,3,4,2)\in\fS_4$ as in \eqref{eqn:m3D4quadric}. Once we have $\varepsilon_3\in G$, the abelian group 
$$
\langle\varepsilon_1,\varepsilon_3, m_3^2\rangle\simeq C_2^3
$$
has no fixed points on $X$. 
If follows from the irreducibility of $V_4$ that $N$ contains an element $\varepsilon_2$ as in \eqref{eqn:varep2}.
When $\mathrm{ord}(t_1)\neq \mathrm{ord}(t_2)$, we know that $\varepsilon_3$ is a power of $\varepsilon_2$. When $t_1^2=t_2^2=1$, $\varepsilon_2=\varepsilon_3$. We are reduced to $$
t_1=\zeta_{2^n}^{2d_1+1},  \quad t_2=\zeta_{2^n}^{2d_2+1}, \quad \text{  for some  } \quad n\geq2, \quad d_1,d_2\in \bZ.  
$$
We have
$$
\varepsilon_2m_3\varepsilon_2^{-1}m_3^{-1}=\mathrm{diag}(1,t_2/t_1,t_1t_2,1/(t_1t_2),t_1/t_2).
$$
Observe that $\mathrm{ord}(t_2/t_1)\neq \mathrm{ord}(t_1t_2)$. Indeed, 
$$
t_2/t_1=\zeta_{2^n}^{2(d_2-d_1)},\quad t_1t_2=\zeta_{2^n}^{2(d_1+d_2+1)},
$$
and $d_2-d_1$ has a different parity from $d_1+d_2+1.$ Thus, either $\varepsilon_3$ or $\varepsilon_1\varepsilon_3$ is a power of $\varepsilon_2m_3\varepsilon_2^{-1}m_3^{-1}$, 
and thus is in $G$. 

\ 

This completes the proof of Theorem~\ref{thm:quad-uni}. 

\ 

We include an alternative proof of Theorem~\ref{thm:quad-uni} which 
requires less computation; it was suggested by the referee.  Recall the Goursat lemma 
from group theory (see, e.g., \cite[Lemma 4.1]{DI}): a subgroup of a product group $G_1\times G_2$ which surjects to $G_1$ and $G_2$ via the natural projections has the form 
$$
G_1\times_HG_2=\{(g_1,g_2)\in G_1\times G_2: \varphi_1(g_1)=\varphi_2(g_2)\},
$$
where $H$ is a group admitting surjective homomorphisms 
$$
\varphi_1:G_1\to H, \quad \varphi_2:G_2\to H.
$$
One can use this fact to simplify the proof of Theorem~\ref{thm:quad-uni}. As above, there are three cases based on the splitting of $V$, and we retain the corresponding notation. Note that we may assume that $\chi_1$ is the trivial character. 

\ 

{\bf\em Case} (1): We have homomorphisms 
$$
\pi_1:G\to\Aut(C), \quad \pi_2:G\to \GL(V_2).
$$
We may assume that $\pi_1(G)=C_2^2$ since otherwise $G$ fixes a point on $C$. Since $\chi_1$ is the trivial character, we know that $\pi_2(G)$ leaves invariant the form $x_4x_5$, which implies that $\pi_2(G)=\fD_{2^{n}}$ for some $n\geq 2$. Similarly, we also have that $\Ker(\pi_1)\cap\Ker(\pi_2)=1$. It follows that 
$$
G=C_2^2\times_H\fD_{2^n},
$$ 
for $H=1,C_2$, or $C_2^2$. If $H=1$ or $C_2$, then $G$ contains a subgroup $C_2^3$ generated by the elements 
$$
(a,1),\quad (1,c), \quad(b,s),
$$
where $a\ne b\in C_2^2$, $c$ is the nontrivial element in the center of $\fD_{2^n}$, and $s\in\fD_{2^n}$ is a reflection (a non-central involution). The fixed locus of $(1,c)$ on $X$ is $C\cup\mathfrak p_2\cup\mathfrak p_3$, but $s$ permutes $\mathfrak p_2\cup\mathfrak p_3$ and $a,b$ act on $C$ via $C_2^2$ with no fixed points. Thus the $G$-action on $X$ fails Condition {\bf (A)}.

\

{\bf\em Case} (2): Similarly, we have homomorphisms $\pi_1:G\to \GL(V_2)$ and $\pi_2:G\to \GL(V_2')$ such that $\pi_1(G)=\fD_{2^{n}}$, $\pi_2(G)=\fD_{2^m}$ for some integers $n,m\geq 1$, and $\Ker(\pi_1)\cap\Ker(\pi_2)=1$. We may assume that $n\leq m$. It follows that 
$$
G\simeq \fD_{2^{n}}\times_H\fD_{2^m},
$$
where $H=1,C_2$, or $\fD_{2^r}$ for some $1\leq r\leq n$. If $H\ne \fD_{2^n}$, then $G$ contains a subgroup $C_2^3$ generated by elements 
$$
(c_1,1),\quad (1,c_2),\quad (s_1,s_2),
$$
where $c_1$ and $c_2$ are in the center of $\fD_{2^n}$ and $\fD_{2^m}$, respectively, and $s_1$ and $s_2$ are reflections in $\fD_{2^n}$ and $\fD_{2^m}$, respectively. The fixed locus of $\langle(c_1,1),(1,c_2)\rangle$ is $\mathfrak p_2\cup\mathfrak p_3\cup\mathfrak p_4\cup\mathfrak p_5$, and $(s_1,s_2)$
switches $\mathfrak p_2\leftrightarrow\mathfrak p_3$, $\mathfrak p_4\leftrightarrow\mathfrak p_5$. The $G$-action fails Condition {\bf (A)}.

When $H=\fD_{2^n}$, the natural surjection $G\to \fD_{2^m}$ is an isomorphism. If $n<m$, then the image of every abelian subgroup of $G$ under the  projection $G\to \fD_{2^n}$ is a cyclic group. This implies that every abelian subgroup of $G$ fixes a point on the conic $C=X\cap\{x_4=x_5=0\}$. Since $\mathrm B_0(G)=0$, we know that  $C$ is $G$-unirational and thus $X$ is $G$-unirational, by Proposition~\ref{prop:quad-genu}.
If $n=m$, then $(c,c)\in G$, where $c$ is the center of $\fD_{2^n}.$ We know that $(c,c)$ acts via $\mathrm{diag}(1,-1,-1,-1,-1)$. Arguing as in the last paragraph of the previous proof of Case (2), we conclude that $X$ is $G$-unirational if Condition {\bf (A)} is satisfied.

\

{\bf\em Case }(3): Recall that $\varepsilon_1\in G$, where $\varepsilon_1= \mathrm{diag}(1,-1,-1,-1,-1)$, and $G'=G/\langle\varepsilon_1\rangle$ acts generically freely on $S=X\cap\{x_1=0\}$. We have 
$$
\Aut(S)=(\PGL_2(k)\times\PGL_2(k))\rtimes C_2.
$$
It follows that 
$$
G=\fD_{2^n}\times_H\fD_{2^m},\quad \text{or}\quad (\fD_{2^n}\times_H\fD_{2^n}).C_2,
$$
where $1\leq n\leq m$, $H=1,C_2, \fD_{2^r}$, for some $r\leq n$. Here the action of $\fD_{2^n}\times\fD_{2^m}$ on $\bP(V_4)$ is given by 
\begin{align*}
    r_1&:(x_2,\ldots,x_5)\mapsto(\zeta_{2^{n+1}}x_2,\zeta_{2^{n+1}}x_3,\zeta_{2^{n+1}}^{-1}x_4,\zeta_{2^{n+1}}^{-1}x_5),\\
    s_1&:(x_2,\ldots,x_5)\mapsto(x_4,x_5,x_2,x_3),\\
    r_2&:(x_2,\ldots,x_5)\mapsto(\zeta_{2^{m+1}}x_2,\zeta^{-1}_{2^{m+1}}x_3,\zeta_{2^{m+1}}x_4,\zeta_{2^{m+1}}^{-1}x_5),\\
    s_2&:(x_2,\ldots,x_5)\mapsto(x_3,x_2,x_5,x_4),
\end{align*}
and the $C_2$ switching  two factors of $S=\bP^1\times \bP^1$ is given by 
  $$ 
  t:(x_2,\ldots,x_5)\mapsto(x_2,x_4,x_3,x_5).
  $$
As above, when $H\ne \fD_{2^n}$, $G'$ contains elements
$$
(c_1,1),\quad (1,c_2), \quad (s_1',s_2'),
$$
where $c_1$ and $c_2$ are in the center of $\fD_{2^n}$ and $\fD_{2^m}$ respectively, and $s_1'$ and $s_2'$ are reflections in $\fD_{2^n}$ and $\fD_{2^m}$ respectively. Let $c_i'$ be a lift of $(c_i,1)$ in $G$, for $i=1,2$, and $s'$ be a lift of $(s_1',s_2')$ in $G$. Then up to multiplying by $\varepsilon_1$, we have that 
$$
    c_1'=\mathrm{diag}(1,\zeta_4,\zeta_4,-\zeta_4,-\zeta_4),\quad   c_2'=\mathrm{diag}(1,\zeta_4,-\zeta_4,\zeta_4,-\zeta_4),
$$
$$
s':(x_1,\ldots,x_5)\mapsto (x_1,ax_5,bx_4,b^{-1}x_3,a^{-1}x_2),
$$
where $a,b$ are roots of unity. It follows that the subgroup $C_2^3$ generated by 
$
(c_1')^2, c_1'c_2'$, and $s'$
has no fixed points on $X$, failing Condition {\bf (A)}.

If $H=\fD_{2^n}$ and $G'=\fD_{2^n}\times_H\fD_{2^m}\simeq \fD_{2^m}$,  then $G'$ is generated by $r_1^ur_2$ and $s_1's_2$, for some odd integer $u$ and $s_1'$ is a reflection in $\fD_{2^n}$. It follows that the line $\{x_1=x_2=x_5=0\}$ is $G$-invariant, contradicting the irreducibility of $V_4$. 

When $H=\fD_{2^n}$ and $G'=(\fD_{2^n}\times_H\fD_{2^n}).C_2$, then $G'$ contains elements $(c,c)$ and $(s_1',s_2')$ where $c$ is in the center of $\fD_{2^n}$, and $s_1'$ and $s_2'$ are reflections. Let $c'$ be the lift of $(c,c)$ in $G$, and $s'$ be the lift of $(s_1',s_2')$ in $G$. As above, the abelian subgroup generated by $c', s'$ and $\varepsilon_1$ does not fix a point on $X$, failing Condition {\bf (A)}.

\section{$G$-unirationality of smooth cubic threefolds}
\label{sect:smooth-3}

In this section, we prove:

\begin{theo}
    \label{thm:main-cubic}
Let $X\subset \bP^4$ be a smooth cubic threefold with a generically free regular action of a finite group $G$. Assume that the action satisfies Condition {\bf (A)}. 
Then $X$ is $G$-unirational, with 
the possible exception of the following actions: 
\begin{itemize}
    \item $G=C_9\rtimes C_3$, and $X$ is the Fermat cubic threefold.
    \item  $G=\PSL_2(\bF_{11})$, or $G= C_{11}\rtimes C_5\subset \PSL_2(\bF_{11})$, and $X$ is the Klein cubic threefold. 
    \item $G=\fA_5$, acting on $X\subset \bP(V_5)$, where $V_5$ is the irreducible 5-dimensional representation. There is a 1-parameter family of such cubics, this is the pencil generated by the Klein and the Segre cubic threefolds.
\end{itemize}
\end{theo}

For applications of results of Section~\ref{sect:cubic} to cubic threefolds, we need to address $G$-unirationality of {\em singular} cubic surfaces, a case not considered in \cite{Duncan}. The proof of the following proposition is suggested by the referee.

\begin{prop}
\label{prop:surface-g}
Let $S\subset \bP^3$ be a normal rational cubic surface, with a generically free regular 
action of a finite group $G$. If the $G$-action satisfies condition {\bf (A)} then it is $G$-unirational.  
\end{prop}
\begin{proof}
    Let $\pi:\tilde S\to S$ be a $G$-equivariant minimal resolution of singularities. Let $S'$ be the output of $G$-MMP applied to $\tilde S$. Then $S'$ is a $G$-minimal surface with $K_{S'}^2\geq 4$. In particular, $S'$ is either a $G$-minimal del Pezzo surface or a $G$-minimal conic bundle. If $S'$ is a del Pezzo surface, it is $G$-unirational by \cite[Theorem 1.4]{Duncan}. If $S'$ is a conic bundle which is not a del Pezzo surface, it is either the Hirzebruch surface $F_2$, or an Iskovskikh surface with two disjoint $(-2)$ curves. The preimage of $(-2)$-curve(s) in $S$ is an orbit of singular points of length 1 or 2. Thus, $S$ is $G$-unirational by Corollary~\ref{coro:2sing}.
\end{proof}
\begin{rema}
One can also prove Proposition~\ref{prop:surface-g} case by case. All combinations of singularities of cubic surfaces are known classically, see \cite{Bruce-Wall}; their automorphisms are partially classified in \cite{sakamaki}.   Examining possible combinations of singularities \cite[Table 1]{sakamaki} and applying Corollary~\ref{coro:2sing}, we obtain $G$-unirationality in all cases, with the possible exception of $3\mathsf A_1$, $3\mathsf A_2$, and $4\mathsf A_1$ singularities. In the last case, the Cayley cubic surface, the action is birational to a linear action, via the standard Cremona involution, see \cite[Section 5]{CTZ-cubic} and \cite[Section 7]{CMTZ}. The $3\sA_2$-cubic surface is toric, and $G$-equivariantly birational to a del Pezzo surface of degree 6, thus $G$-unirational when it satisfies Condition {\bf (A)}, by \cite{Duncan}.  

Any $3\sA_1$-cubic surface $S$ with an action of a subgroup $G\subseteq \Aut(S)$ is $G$-unirational. 
Indeed, we may assume the singular points of $S$ are 
    $$
    [1:0:0:0],\quad [0:1:0:0],\quad[0:0:1:0].
    $$
    By linear algebra, $S$ is given by
\begin{align}
\label{eqn:3a1eq}
\{x_1x_2x_3+(x_1+x_2+ax_3)x_4^2+x_4^3=0\},
\end{align}
for some $a\ne 0$, with $\Aut(S)=\fS_3$, generated by 
\begin{align*}
\sigma_1&:(x_1,x_2,x_3,x_4)\mapsto(x_2,x_1,x_3,x_4),\\
\sigma_2&:(x_1,x_2,x_3,x_4)\mapsto(ax_3,x_1,\frac{x_2}{a},x_4).
\end{align*}
Indeed, it is clear that $\sigma_1,\sigma_2\in\Aut(S)$. 
From the equation, one sees that there are no nontrivial elements in $\Aut(S)$ fixing all three nodes. Therefore, $\Aut(S)=\langle\sigma_1,\sigma_2\rangle$.
Furthermore, $\Aut(S)$ fixes three smooth points on $S$, implying that $S$ is $\Aut(S)$-unirational, by Proposition~\ref{prop:ind-1}.
 \end{rema}

We also need the following lemmas.
\begin{lemm}\label{lemm:cubsurfgp}
    Let $S\subset \bP^3$ be a normal rational cubic surface, with a generically free regular action of a finite abelian group $G$. If $G$ does not fix any point on $S$, then $G\simeq C_3^2$ or $C_3^3$.
\end{lemm}
 \begin{proof}
    The $G$-action on $\bP^3$ lifts to $\GL_4(k)$, and thus we may assume that it is diagonal. When $G$ does not fix points on $S$, we know that the defining equation of $S$ in $\bP^3_{x_1,\ldots,x_4}$ contains terms $x_1^3,x_2^3,x_3^3,x_4^3$, since otherwise one of the coordinate points 
    $$
    [1:0:0:0],\quad [0:1:0:0],\quad [0:0:0:1],\quad [0:0:0:1]
    $$
    is a $G$-fixed point on $S$. It follows that any element  $g\in G$ has order $3$, i.e., $G\simeq C_3, C_3^2,C_3^3$. A cyclic group action on $S$ always has fixed points. Thus $G\simeq C_3^2$ or $C_3^3$.
 \end{proof}

 \begin{lemm}\label{lemm:invariantunicub}
Let $X\subset \bP^4$ be a smooth cubic threefold threefold with a generically free regular action of a finite group $G$. Assume that $G$ does not contain a subgroup isomorphic to $C_3^2$, and there is a $G$-invariant hyperplane section $S$ in $X$. Then the $G$-action on $X$ is equivariantly unirational.      
 \end{lemm}
\begin{proof}
Since $X$ is smooth, we know that $S$ is normal and not a cone. When $S$ has no generic stabilizer, it follows from Lemma~\ref{lemm:cubsurfgp} that the $G$-action on $S$ satisfies Condition {\bf (A)}, and thus is equivariantly unirational by Proposition~\ref{prop:surface-g}. From Proposition~\ref{prop:cubic-uni-2}, the $G$-action on $X$ is equivariantly unirational.

When $S$ has a generic stabilizer, necessarily a cyclic group $C_n$.  We may assume that 
$$
S=X\cap\{x_1=0\}\subset \bP^4_{x_1,\ldots,x_5},
$$
and $G$ acts trivially on $x_1$. We know that $C_n$ is a subgroup of the center of $G$.  The group $\bar G=G/C_n$ acts generically freely on the cubic surface
$
S\subset \bP^3_{x_2,\ldots,x_5}.
$ It follows that the $\bar G$-action lifts to $\GL_4(k)$, i.e., $G$ is a split central extension of $\bar G$ by $C_n$. Thus $G\simeq C_n\times \bar G$. By assumptions, $\bar G$ does not contain a subgroup isomorphic to $C_3^2$. Then the $\bar G$-action on $S$ is equivariantly unirational by Proposition~\ref{prop:surface-g}, and thus the $G$-action on $S$ is equivariantly unirational. Proposition~\ref{prop:cubic-uni-2} implies that the $G$-action on $X$ is equivariantly unirational.
\end{proof}
The rest of this section is devoted to a proof of Theorem~\ref{thm:main-cubic}. 
Actions of finite groups on smooth cubic threefolds $X\subset \bP^4$ have been classified by Wei and Yu in \cite{wei-yu}. More precisely, they find all groups that can act, and give examples of cubic threefolds that realize such actions, \cite[Theorem 1.1 and Section 3]{wei-yu}. 
We recall the maximal nonabelian groups, and the unique cubic threefolds admitting such actions, in the notation of \cite{wei-yu}:

\begin{table}[H]
\begin{tabular}{|c|c|c|}
\hline &&\\[-0.4cm]
        & Cubic & Automorphism\\[-0.5cm] && \\\hline &&\\[-0.3cm]
    $X_1$ & $\sum_{i=1}^5 x_i^3=0$ & $C_3^4\rtimes \fS_5$\\[-0.3cm] && \\\hline &&\\[-0.3cm]
  $X_2$&$ 3(\sqrt{3}-1)x_1x_2x_3+\sum_{i=1}^5 x_i^3=0$ & $((C_3^2\rtimes C_3)\rtimes C_4)\times \fS_3$\\[-0.3cm] && \\\hline &&\\[-0.3cm]
$X_5$ & $x_1^2x_2+x_2^2x_3+x_3^2x_4+x_4^2x_5+x_5^2x_1=0$&$\PSL_2(\mathbb F_{11})$\\[-0.3cm] && \\\hline &&\\[-0.3cm]
        $X_6$ &$\sum_{i=1}^6x_i^3=\sum_{i=1}^5x_i=0$&$C_3\times \fS_5$\\[-0.3cm] && \\\hline
\end{tabular}
\vskip 0.2cm
\caption{}\label{table:cubic}
\end{table}

Using {\tt magma}, we find that there are 82 {\em nonabelian} subgroups of the maximal groups appearing in Table~\ref{table:cubic} in total. 

A table of abelian actions is also in \cite[Appendix B]{wei-yu}. 
In particular, they provide a list of all abelian groups that admit a unique action on a smooth cubic threefold. There are 12 isomorphism classes of such groups. We compile a list of nonabelian groups which contain one of these abelian groups in \cite{CTZuni-tables}: these groups also admit unique actions on smooth cubic threefolds. 
 There are 41 such groups, up to isomorphism. We note that several of them admit nonconjugated actions on the same cubic -- there are 50 conjugacy classes of actions. All the actions are realized on one of the cubic threefolds in Table~\ref{table:cubic}.

Now, we study the 82 isomorphism classes of nonabelian subgroups. We divide them into three types:
\begin{enumerate}
    \item groups whose orders are divisible by 5, 
    \item groups where their orders are not divisible by 5 and do not contain a subgroup isomorphic to $C_3^2$,
    \item the remaining groups.
\end{enumerate}

{\em{Type (1):}} First we consider groups $G$ where $5\mid|G|$. These groups can possibly have 5-dimensional irreducible representations. The isomorphism classes are 
\begin{itemize}
    \item $\fD_5,\fF_5,C_3\times\fD_5$: the corresponding cubics are $G$-unirational by Proposition~\ref{prop:ind-2}, since in each case, $G$ contains an index-2 cyclic subgroup which fixes a point on the cubic.
    \item $\fC_3^4\rtimes C_5, C_3^4\rtimes\fD_5, C_3^4\rtimes\fF_5, C_3^4\rtimes\fA_5, C_3^4\rtimes\fS_5$: these groups contain $C_3^4,$ which only acts on the Fermat cubic threefold with no fixed points. The corresponding actions are not equivariantly unirational. 
    \item $ C_3\times \fF_5, C_3\times \fA_5, C_3\times \fS_5 $: these groups act on a unique smooth cubic threefold, $X_6$. Each action leaves invariant a hyperplane section, which is isomorphic to the Clebsch cubic surface, with generic stabilizer $C_3$, and a residual action of $\fS_5$, $\fA_5$, and $\fF_5$, respectively. These actions on the Clebsch cubic surface are unirational. By Proposition~\ref{prop:cubic-uni-2}, the respective action on $X_6$ is unirational.
    \item $\fS_5$: from representation theory, one can check that any $\fS_5$-action on a smooth cubic threefold leaves invariant a hyperplane section. The action is equivariantly unirational by Lemma~\ref{lemm:invariantunicub}.
    \item  $\fA_5$: from representation theory, an $\fA_5$-action on a smooth cubic threefold either leaves invariant a hyperplane section, or is irreducible on $\bP^4$. In the former case, the action is equivariantly unirational by Lemma~\ref{lemm:invariantunicub}. In the latter case, its equivariant unirationality remains open. There is a unique such action on $\bP^4$, and a pencil of invariant cubic threefolds such that a general member of the pencil is smooth.  
    \item $C_{11}\rtimes C_5, \PSL_2(\bF_{11})$: these groups act on a unique smooth cubic threefold, the Klein cubic threefold $X_5$. Their equivariant unirationalities are open. 
\end{itemize}

\ 

{\em{Type (2):}}  Next, we turn to groups $G$ where $5\nmid |G|$ and $G$ does not contain a subgroup isomorphic to $C_3^2$. We show that such actions are always equivariantly unirational.

\begin{lemm}
    Let $X\subset \bP(V)=\bP^4$ be a smooth cubic threefold with a generically free regular action of a finite group $G$, where $V$ is a linear representation of $G$. Assume that $5\nmid |G|$ and $G$ does not contain a subgroup isomorphic to $C_3^2$. Then the $G$-action on $X$ is equivariantly unirational.
\end{lemm}
\begin{proof}
    Since $5\nmid |G|$, we know that $V$ is reducible. If $V$ contains a 1-dimensional subrepresentation, the $G$-action on $X$ is equivariantly unirational by Lemma~\ref{lemm:invariantunicub}. It suffices to consider when $V=V_2\oplus V_3$ where $V_2$ and $V_3$ are irreducible 2 and 3-dimensional representations of $G$.  Thus $l=\bP(V_2)$ is a invariant line in $\bP^4$. If $l\subset X$, then the $G$-action on $X$ is equivariantly unirational by Proposition~\ref{prop:cubic-uni-2}. When $l\cap X$ consists of a $G$-orbit of 1 or 2 points, the $G$-action on $X$ is equivariantly unirational by Proposition~\ref{prop:ind-2}. When $l\cap X$ consists of a $G$-orbit of 3 points, $G$ acts on $l$ via $C_3$ or $\fS_3$. There exists a subgroup $G'\subseteq G$ of index 1 or 2 respectively such that $G'$ fixes a point on $l$, i.e., $G'$ leaves invariant a hyperplane section on $X$. It follows from Lemma~\ref{lemm:invariantunicub} that the $G'$-action on $X$ is equivariantly unirational. Corollary~\ref{coro:2sing} then implies that the $G$-action on $X$ is equivariantly unirational. 
\end{proof}

\

{\em{Type (3):}}  We consider the remaining groups, i.e., those groups $G$ containing a subgroup isomorphic to $C_3^2$ and $5\nmid |G|$. There are 66 isomorphism classes. Among them, 33 groups have a normal abelian subgroup $G'$ such that $G/G'$ is a 2-group. Actions of these groups on smooth cubic threefolds are equivariantly unirational if Condition {\bf (A)} is satisfied,  by Proposition~\ref{prop:ind-2}.

Among the remaining 33 isomorphism classes, there are 20 groups which act on a unique cubic threefold: 17 of them act on the Fermat cubic $X_1$, and 3 of them act on $X_2$. For actions of these 20 groups $G$ on the corresponding cubics $X$, we check via computations that one of the following holds: 
\begin{itemize}
    \item the $G$-action fixes a point on $X$,
    \item the $G$-action on $X$ fails Condition {\bf (A)},
    \item $X=X_1$ is the Fermat cubic threefold and $G=C_9\rtimes C_3$ is generated by 
    $$
    \mathrm{diag}(1,1,\zeta_3,\zeta_3^2,1),  \quad  \mathrm{diag}(1,1,1,\zeta_3,\zeta_3^2), \quad(x_1,\zeta_3x_2,x_4,\zeta_3x_5,x_3).
    $$
\end{itemize}
The equivariant unirationality of the last action remains open. 

\

There are 13 isomorphism classes left:

   {\small $$C_3\times \mathrm{He}_3,\,\,\, C_3\times C_3^2\rtimes \fS_3, \,\,\, \mathrm{He}_3\rtimes \fS_3, \,\,\, \fS_3\times \mathrm{He}_3,\quad  \fS_3\times C_3^2\rtimes \fS_3, \,\,\, (C_3\times\mathrm{He}_3)\rtimes C_4,$$
 $$\mathrm{He}_3,\quad C_3^2\rtimes \fS_3,\quad C_2\times\mathrm{He}_3,\quad C_2\times C_3^2\rtimes \fS_3,\quad C_3^2\rtimes \fS_3.C_2,$$
$$C_3\times\fA_4,\quad C_3\times \fS_4,$$}

All groups in the first row contain a subgroup isomorphic to $C_3\times\mathrm{He}_3$; all groups $G$ in the second row contain a normal subgroup isomorphic to $\mathrm{He}_3$ and $G/\mathrm{He}_3$ is a 2-group.  

We focus on four groups: $G=\mathrm{He}_3$, $C_3\times\mathrm{He}_3$, $C_3\times\fA_4$ and $C_3\times\fS_4$. Using {\tt magma}, we can find all $G$-actions on $\bP^4$, and their invariant cubics, see  \cite{CTZuni-tables}.
\begin{itemize}
    \item When $G=\mathrm{He}_3$, we check that a $G$-action on a smooth cubic threefold either fixes a point, or fails Condition {\bf (A)}. Thus the $G$-action in equivariantly unirational if and only if Condition {\bf (A)}  is satisfied. It follows that the same holds for all groups in the second row above.
    \item When $G=C_3\times \mathrm{He}_3$, we check that all $G$-actions on smooth cubic threefolds fail Condition {\bf (A)}, and thus the same holds for all groups in the first row above.
    \item  When $G=C_3\times \fA_4$ or $C_3\times\fS_4$, we check that for any $G$-action on a smooth cubic threefold, there are two invariant hyperplanes. One of them has generic stabilizer $C_3$ and the $G/C_3$-action on the corresponding cubic surface $S$ satisfies Condition {\bf (A)}. The $G$-actions on both $S$ and the cubic threefold are equivariantly unirational.
\end{itemize} 

This completes the proof of Theorem~\ref{thm:main-cubic}.
\begin{rema} 
Assuming \cite[Conjecture 10.4 and Theorem 10.5]{DR}, $G$-unirationality of $X$ {\em would follow} from $G_3$-unirationality of $X$, where $G_3$ is a 3-Sylow subgroup of $G$. 
There are six nonabelian 3-groups contained in Table~\ref{table:cubic}: 
$$
\mathrm{He}_3, C_9\rtimes C_3, C_3\wr C_3, C_3\times\mathrm{He}_3, (C_3\times C_9)\rtimes C_3, C_3\times(C_3\wr C_3).
$$ 
The following, obtained via {\tt magma} computations,  {\em would} imply that all $G$-actions on $X$, with the possible exception of $C_9\rtimes C_3$, are $G$-unirational:
Let $X$ be a smooth cubic threefold and $G\subset\Aut(X)$ a 3-group. Then: 
\begin{itemize}
    \item the $G$-action fixes a point on $X$, or
    \item the $G$-action fails condition {\bf (A)}, or
    \item $X$ is the Fermat cubic threefold and $G=C_9\rtimes C_3$ is generated by 
    $$
    \mathrm{diag}(1,1,\zeta_3,\zeta_3^2,1),  \quad  \mathrm{diag}(1,1,1,\zeta_3,\zeta_3^2), \quad(x_1,\zeta_3x_2,x_4,\zeta_3x_5,x_3).
    $$
\end{itemize}

\end{rema}

\section{$G$-unirationality of singular cubic threefolds}
\label{sect:sing}

In this section, we prove: 

\begin{theo}\label{thm:singcub-main}
    Let $X$ be a cubic threefold with isolated singularities which is not a cone, carrying a generically free regular action of a finite group $G$. Assume that the $G$-action on $X$ satisfies Condition {\bf (A)}. Then $X$ is $G$-unirational.
\end{theo}

The proof is based on the classification of configurations of singularities and possible group actions in \cite{CTZ-burk}, \cite{CTZ-cubic}, \cite{CMTZ}. In most of the cases, we can apply Propositions~\ref{prop:ind-1}, \ref{prop:cubic-uni-2}, and Corollary~\ref{coro:2sing}. The remaining cases, 
\begin{itemize}
    \item $3\sA_1,\quad$ $G$ is a finite subgroup of $\Aut(X)=\GL_2(\bF_3)$,
    \item $3\sD_4,\quad$  $G$ is a finite subgroup of $\Aut(X)=(\bG_m^2\times \fS_3)\rtimes\fS_3$, 
\item $6\sA_1,\quad G=\fS_4, \fS_5$,
\end{itemize}
require a separate analysis. 


The linearizability properties of singular cubic threefolds have been studied in \cite{CTZ-cubic} and \cite{CMTZ}; there are linearizable actions, as well as actions failing stable linearizability. The classification results of that paper allow us to address $G$-unirationality.

By Corollary~\ref{coro:2sing}, cubic threefolds $X$ with a distinguished singularity or with the following singularity types are $G$-unirational, for all $G\subseteq \Aut(X)$: 
$$
2\sA_n,n=1,\ldots,5,\quad  2\sD_4, \quad  2\sA_2+2\sA_1, \quad 2\sA_3+2\sA_1, 
$$
$$
2\sD_4+2\sA_1, \quad 2\sA_2+3\sA_1, \quad  2\sA_3+3\sA_1, \quad 3\sA_2+2\sA_1, \quad 2\sD_4+3\sA_1.
$$
We summarize the analysis of the remaining cases:

\subsection*{$3\sA_n, 3\sD_4, n=1,2,3$}
With {\tt magma}, we verify that in all cases, except 
\begin{itemize}
    \item[(1)] $3\sA_1,\quad$  $G\subseteq \GL_2(\bF_3)$,
    \item[(2)] $3\sD_4,\quad $ $G$ is a finite subgroup of $\Aut(X)=(\bG_m^2\times \fS_3)\rtimes\fS_3$, 
\end{itemize}
the cubic threefold $X$ contains a $G$-invariant 
$G$-unirational hyperplane section $S\subset X$, a (possibly singular) cubic surface. Indeed, we check that $S$ has fixed points upon restriction of the action to abelian subgroups of the quotient of $G$ by the generic stabilizer of $S$. Proposition~\ref{prop:surface-g} implies that $S$ is equivariantly unirational for the action of the quotient, and thus for the action of $G$. 
It then follows from Proposition~\ref{prop:cubic-uni-2} that $X$ is $G$-unirational.

\ 

{\em Case} 1:  $X\subset \bP(V)$ is the cubic threefold with $3\sA_1$-singularities given by 
$$
x_1x_2x_3+x_1q_1+x_2q_2+x_3q_3=0,
$$
where \begin{align*}
    q_1&=x_4^2+bx_4x_5+x_5^2,\\
    q_2&=x_4^2+\zeta_3bx_4x_5+\zeta_3^2x_5^2,\\
    q_3&=x_4^2+\zeta_3^2bx_4x_5+\zeta_3x_5^2,
\end{align*}
for $b^2=\sqrt{-2}$, and $V=V_2\oplus V_3$, for irreducible representations $V_2$ and $V_3$ of $G=\GL_2(\bF_3)$ generated by 
\begin{align*}
    &\mathrm{diag}(1,1,1,-1,-1),\\
(x_1,\ldots,x_5)\mapsto&(x_2,x_3,x_1,x_4,\zeta_3x_5),\\(x_1,\ldots,x_5)\mapsto&(\zeta_3x_2,\zeta_3^2x_1,x_3,x_4,\zeta_6x_5),\\
(x_1,\ldots,x_5)\mapsto&(x_2,\zeta_6^5x_1,\zeta_6x_3,\frac{\zeta_6bx_4+x_5}{1-\zeta_3},\frac{\zeta_6x_4+bx_5}{1-\zeta_3}).
\end{align*}
The $G$-invariant line $\mathfrak l =\bP(V_2)=\{x_1=x_2=x_3=0\}$ is contained in the smooth locus of $X$. Note that $\mathfrak l$ is $G$-unirational, but the $G$-action is not generically free on $\mathfrak l$.   The 
$G$-unirationality of $X$ then follows from Proposition~\ref{prop:cubic-uni-2}.

\

{\em Case} 2: $X$ is given by 
\begin{equation} 
\label{eqn:cubic-123}
x_1x_2x_3+x_4^3+x_5^3=0
\end{equation}
and 
$$
\Aut(X)=N\times \fS_3, \quad N:=\bG_m^2\rtimes\fS_3, 
$$
where $N$ is 
acting via $\fS_3$-permutation and torus actions on $x_1,x_2,x_3$, and $\fS_3$ via the irreducible 2-dimensional representation on $x_4,x_5$. We use notation from \cite[Proposition 6.6]{CMTZ}
\begin{align}\label{eqn:3D4aut}
    N=\langle \tau_{a,b}, \sigma_{(123)}, \sigma_{(12)}\rangle,\quad 
\fS_3=\langle \eta, \sigma_{(45)}\rangle, 
\end{align}
where
\begin{align*}
    \tau_{a,b}&:\mathrm{diag}(a,b,(ab)^{-1},1,1),\\
    \eta&:\mathrm{diag}(1,1,1,\zeta_3,\zeta_3^2),
\end{align*}
and
$\sigma$ are permutations of the variables as is labeled.
Consider the subgroup 
\begin{align}\label{eqn:3D4H}
H:=\langle \tau_{a,b}, \sigma_{(123)}\cdot \eta\rangle\subset \Aut(X). 
\end{align}
By \cite[Proposition 6.6]{CMTZ}, $X$ is $H$-equivariantly birational to $S\times \bP^1$, where $S$ is the del Pezzo surface of degree 6. In particular, for every $G\subset H$ satisfying Condition {\bf (A)}, we obtain $G$-unirationality of $X$.

\begin{lemm}
Let $X$ be the cubic \eqref{eqn:cubic-123}, 
$G\subset \Aut(X)$ a finite subgroup. Then $X$ is $G$-unirational if and only if Condition {\bf (A)} is satisfied. 
\end{lemm}

\begin{proof}
As discussed, Condition {\bf (A)} is necessary for $G$-unirationality. We now assume that the $G$-action on $X$ satisfies Condition {\bf (A)} and proceed to show that $X$ is $G$-unirational.

Observe that there are two distinguished $G$-orbits of three points 
$$
P_1=\{[1:0:0:0:0],[0:1:0:0:0],[0:0:1:0:0]\},
$$
$$
P_2=\{[0:0:0:1:-1],[0:0:0:1:-\zeta_3],[0:0:0:1:-\zeta_3^2]\}. 
$$
Using Propositions~\ref{prop:ind-1} and \ref{prop:ind-2}, we are reduced to the case when $G$ acts on both $P_1$ and $P_2$ via $C_3$. Up to conjugation in $\Aut(X)$, $G$ contains
$$
m_1:=\sigma_{(123)}\cdot \eta.
$$
The $G$-action on $P_1$ induces a homomorphism $G\to C_3$. Let $N$ denote its kernel so that we have
$$
G=\langle N,m_1\rangle.
$$

\noindent 
If $N\subset \bG_m^2$, then $G\subset H$, from  \eqref{eqn:3D4H}, and $X$ is $G$-unirational. 

\ 

\noindent 
If $N\not\subset \bG_m^2$, then $N$ contains 
$$
m_2:=\tau_{a,b}\cdot\eta,
$$
for some $\tau_{a,b}\in\bG_m^2$, and $G$ is generated by 
$m_1$, $m_2$, and a subgroup 
$$
M=\bG_m^2\cap G.
$$
Up to replacing $m_2$ by its power, we may assume that $\mathrm{ord}(\tau_{a,b})$ is a power of $3$, so that both $a$ and $b$ are 3-power roots of unity. Condition {\bf (A)} implies that $G$ does not contain either of 
$$
\eta \quad\text{or}\quad\tau_{\zeta_3,\zeta_3}=\mathrm{diag}(\zeta_3,\zeta_3,\zeta_3,1,1),
$$
since both abelian groups
$$
\langle m_1,\eta\rangle\simeq\langle m_1,\tau_{\zeta_3,\zeta_3}\rangle\simeq C_3^2
$$ 
do not fix points on $X$. In particular, $(a,b)\ne (1,1)$. 

 If $|M|$ is not coprime to $3$, then $M$ contains either $\tau_{1,\zeta_3}$ or $\tau_{\zeta_3,1}$, which is impossible since
$$
\tau_{1,\zeta_3}\big(m_1\cdot\tau_{1,\zeta_3}\cdot m_1^{-1}\big)^2=\tau_{\zeta_3,1}\big(m_1^2\cdot\tau_{\zeta_3,1}\cdot m_1^{-2}\big)^2=\tau_{\zeta_3,\zeta_3}\not\in G.   
$$
Hence, $|M|$ is coprime to $3$, the order of $\tau_{a,b}$ must be $3$, and
$$
\tau_{a,b}\in\big\{\tau_{1,\zeta_3},\tau_{1,\zeta_3^2},\tau_{\zeta_3,1},\tau_{\zeta_3^2,1},\tau_{\zeta_3,\zeta_3},\tau_{\zeta_3^2,\zeta_3^2}\big\}.
$$
Conjugating $\tau_{a,b}$ by $m_1$, we may assume that 
\begin{itemize}
\item either $\tau_{a,b}=\tau_{\zeta_3,\zeta_3}$ and $m_2=\tau_{\zeta_3,\zeta_3}\cdot\eta$, or
\item $\tau_{a,b}=\tau_{\zeta_3^2,\zeta_3^2}$ and $m_2=\tau_{\zeta_3^2,\zeta_3^2}\cdot\eta$.
\end{itemize}
Without loss of generality, we may assume that  $m_2=\tau_{\zeta_3,\zeta_3}\cdot\eta$. Then $G$ is generated by $m_1=\sigma_{(123)}\cdot \eta$, $m_2$, and a subgroup 
$M$ of order coprime to $3$.

Let $S:=X\cap\{x_5=0\}$, and $Q:=\langle m_1, M\rangle$. Then $S$ is a cubic surface with $3\sA_2$-singularities, with a generically free action of $Q$ and generic stabilizer $m_2$, i.e., 
$S=X^{\langle m_2\rangle}$. Also note that $Q\simeq G/\langle m_2 \rangle$. The $Q$-action on $S$ satisfies Condition {\bf (A)}. Indeed, if $A\subset Q$ is an abelian subgroup that does not fix any points on $S$, then $\langle m_2, A\rangle\subset G$ is an abelian subgroup that does not fix points in $X$. By Proposition~\ref{prop:surface-g}, the $Q$-action, and thus the $G$-action on $S$ is unirational, which implies that the $G$-action on $X$ is unirational, by Proposition~\ref{prop:ind-1}. The $Q$-unirationality of $S$ also follows from the fact that $S$ is $Q$-birational to $\mathbb{P}^2$ with a linear action of $Q$, combined with the classification of all possibilities for $Q$ given in \cite[Theorem~4.7]{DI}. 
\end{proof}

\subsection*{$4\sA_n, n=1,2$, in general position}
The singular points are in linear general position. From the equations in \cite[Section 5]{CTZ-cubic} and \cite[Section 7]{CMTZ}, we see that $X\cap \{ x_5=0\}$ is a $G$-invariant Cayley cubic surface, which is $G$-linearizable. 
Proposition~\ref{prop:cubic-uni-2} implies the $G$-unirationality of $X$.

\subsection*{$4\sA_1$, contained in a plane}
Let $\Pi\subset X$ be the plane containing the four nodes. 
Such $X$ are equivariantly birational to a smooth intersection of two quadrics, discussed in Section~\ref{sect:2quad}, with their automorphisms and normal forms. 
Proposition~\ref{prop:cubic-uni-2} and Remark~\ref{rema:canbe} imply the $G$-unirationality of $X$.

\subsection*{$5\sA_n, n=1,2$}
Cubics in this case are equivariantly birational to a smooth quadric threefold. The action is either a subgroup of the $\fS_5$-permutation action on coordinates, or a subgroup of the irreducible representation from $\fA_5$, which is stably linearizable by Theorem~\ref{thm:quad-uni}, see also \cite[Proposition 6.2]{HT-quadric}.

\subsection*{$6\sA_1$} 
The classification of possible actions in \cite[Section 7]{CTZ-cubic} yields subcases:
    \begin{itemize} 
    \item[(1)] {\bf No plane:} $\Aut(X)$ is one of the following
    $$
C_2, \,\,\fS_3, \,\,\fS_4, \,\,\mathfrak D_4, \,\,C_2^2, \,\,\mathfrak D_6, \,\,\fS_3^2\rtimes C_2, \,\,\fS_5. 
    $$
All groups failing  Condition {\bf (A)} contain the unique subgroup 
$C_3^2\subset \fS_3^2\rtimes C_2$; they are: 
$$
C_3^2, \quad 
C_3\rtimes\fS_3, \quad 
C_3\times\fS_3, \quad  
\fS_3^2,  \quad 
C_3\rtimes \fS_3.C_2, \quad 
\fS_3^2\rtimes C_2.
$$

\item[(2)]
{\bf One plane:} all actions are linearizable, see \cite[Section 7]{CTZ-cubic}.
\item[(3)] 
{\bf Three planes:}  $\Aut(X)$ is one of the following
$$
C_2^2, \quad C_2^3,\quad C_2\times\fS_3, \quad C_2\times\fS_4,
$$
where $X$ is given by 
$$
x_2x_3x_4+ax_1^3+x_1^2(b_1x_2+b_2(x_3+x_4))+x_1(x_2^2+x_3^2+x_4^2-x_5^2)=0,
$$
for some $a,b_1,b_2\in k$. 
\end{itemize}

In Case (3), the point $[0:0:0:0:1]\in X$ is fixed by any $G\subseteq \Aut(X)$; thus, $X$ is $G$-unirational, by 
Proposition~\ref{prop:ind-1}.

We turn to Case (1).
Note that the action on $X$ is not stably linearizable if 
the invariant class group has rank one, by \cite[Proposition 7.5]{CTZ-cubic}. To prove unirationality, it suffices to consider
the case when $\mathrm{Cl}^{G}(X)=\bZ^2$, by passing to a suitable subgroup of index two, using Proposition~\ref{prop:ind-2}. 

This $G$ leaves invariant each class of cubic scrolls on $X$. 
Recall from \cite[Diagram 7.1]{CTZ-cubic} that an invariant divisor class of cubic scrolls shows that there exists a $G$-equivariant small resolution $X^+\to X$ such that $X^+$ is a weak Fano 3-fold that admits a $G$-equivariant $\mathbb{P}^1$-bundle structure
$$
p^+: X^+\to \bP^2,
$$
whose fibers are mapped to lines in the cubic $X$. 
Note that 
$$
X^+\simeq\bP(\mathcal E),
$$
where $\mathcal{E}$ is a vector bundle of rank 2 on $\mathbb{P}^2$,  described in  \cite[Theorem~3.2(7)]{langer}. The weak Fano 3-fold $X^+$ also appeared in \cite[Theorem~3.6]{Jahnke}. 

Note that $\rH^1(G,\Pic(\mathbb{P}^2))=0$ and $\mathrm{B}_0(G)=0$, for all $G$ appearing in this case; the smallest $G$ with nontrivial Bogomolov multiplier has order $|G|=p^5$, see, e.g., \cite{Mora}. By Lemma~\ref{lemm:pn}, and Condition {\bf (A)} for $X$, the base $\mathbb{P}^2$ is $G$-unirational. 
Lemma~\ref{lemm:line} then implies that the $G$-action lifts to $\mathcal E$. 
Note that the $G$-action on the base $\bP^2$ may have generic stabilizers. Nevertheless, applying the no-name lemma as in Remark~\ref{rema:explicitcub} yields $G$-equivariant birationality 
$$
\mathcal E\times V  \sim_G \mathbb{P}^2\times V\times \bA^2, 
$$
where $V$ is a faithful $G$-representation, with trivial action on $\bA^2$. This implies $G$-unirationality of $X$. 

\subsection*{$8\sA_1$}
In all  cases, there is a $G$-invariant Cayley cubic surface $S\subset X$ (given by $x_3=0$ in the equations in \cite[Proposition 8.2]{CTZ-cubic}), with a $G$-fixed point on $S$. It follows that $X$ is $G$-unirational  by Proposition~\ref{prop:ind-1}.

\subsection*{$9\mathsf A_1$} 
We recall the normal form, see \cite[Section 9]{CTZ-cubic}:
$$
x_1x_2x_3-x_4x_5x_6=a(x_1+x_2+x_3)+x_4+x_5+x_6=0,
$$
for $a^3\ne 0,-1$,
with $G$ a subgroup of $\fS_3^2$, when $a^3\neq 1$, and of 
$\fS_3^2\rtimes C_2$, when $a=1$. 
There is an invariant smooth cubic surface $S\subset X$, given by
$$
x_1+x_2+x_3=0.
$$
When the $G$-action on $X$ satisfies Condition {\bf (A)}, then so does the $G$-action on $S$. Thus $S$ is $G$-unirational, and we apply Proposition~\ref{prop:cubic-uni-2}.

\subsection*{$10\mathsf A_1$} We have a dominant rational map $\mathrm{Gr}(2,6)\to X$, the Segre cubic threefold, equivariantly for the action of $G=\fS_6=\Aut(X)$. By \cite[Proposition 19]{HT-torsor}, the $\fS_6$-action on 
    $\mathrm{Gr}(2,6)$ is stably linearizable, thus $X$ is $G$-unirational.

\section{Intersections of two quadrics}
\label{sect:2quad}

In this section, we prove:

\begin{theo}
\label{thm:2-2}
    Let $X\subset \bP^5$ be a smooth complete intersection of two quadrics and $G\subseteq \Aut(X)$. The following are equivalent: 
    \begin{itemize}
     \item the $G$-action on $X$ is unirational,
        \item the $G$-action on $X$ satisfies condition {\bf (A)}, i.e., every abelian subgroup of $G$ fixes a point on $X$,
          \item the $G$-action fixes a point on $X$.
    \end{itemize}
\end{theo}

The problem of classification of actions of finite groups $G$ on 
$$
X=X_{2,2}=Q_1\cap Q_2\subset \bP^5,
$$
a smooth complete intersection of two quadrics, has been addressed in \cite{avilov}. There is a homomorphism $\pi: G\to \fS_6$, where the image is the induced action of a quotient of $G$ on the pencil $\lambda_1Q_1+\lambda_2Q_2$.  We write down equations in each case:  

\begin{itemize}
    \item Type I, $\fS_4$: $\sum_{j=0}^4x_j^2 =x_0^2+\zeta_4x_1^2-x_2^2-\zeta_4x_3^2+x_5^2= 0$,
    \item Type II, $\fD_6$: $\sum_{j=0}^5x_j^2 =\sum_{j=0}^5\zeta_{6}^{j}x_j^2= 0 $,
      \item Type III, $\fS_3$: $\sum_{j=0}^5x_j^2 =a(x_0^2+\zeta_3x_1^2+\zeta_3^2x_2^2)+x_3^2+\zeta_3x_4^2+\zeta_3^2x_5^2=0$,
       \item Type IV, $C_5$: $\sum_{j=0}^5x_j^2 =\sum_{j=0}^4\zeta_5^{j}x_j^2=0$,
      \item Type V, $C_2^2$: $\sum_{j=0}^4x_j^2 =a(x_0^2-x_1^2)+x_2^2-x_3^2+x_5^2=0$,
      \item Type VI, $C_2$: $\sum_{j=0}^5x_j^2 =a(x_0^2-x_1^2)+b(x_2^2-x_3^2)+x_4^2-x_5^2=0$.
\end{itemize}

We proceed with a list of maximal subgroups $G\subseteq \Aut(X)$ satisfying Condition {\bf (A)}, up to conjugation in $\Aut(X)$. In each type, we list $\pi(G)$, the sign changes $\bar G=\ker(\pi)$, and indicate the number of conjugacy classes of each group in bracket.

\begin{table}[H]
\centering
\small

\begin{minipage}[t]{0.48\textwidth}
\centering
\begin{tabular}{|c|c|c|c|}
\hline&&&\\[-0.4cm]
Type& $G$ & $\bar G$&$\pi(G)$\\
\hline &&&\\[-0.4cm]

I&$\fD_4(3)$&$C_2^2$&$C_2$\\\hline&&&\\[-0.4cm]

I&$C_2^2\rtimes C_4(1)$&$C_2^3$&$C_2$\\\hline &&&\\[-0.4cm]

I&$C_2\times \fA_4(1)$&$C_2^3$&$C_3$\\\hline&&&\\[-0.4cm]

I&$C_2\times C_8(1)$&$C_2^2$&$C_4$\\\hline&&&\\[-0.4cm]

I&$\fS_3(2)$&$1$&$\fS_3$\\\hline
\end{tabular}
\end{minipage}
\hfill
\begin{minipage}[t]{0.48\textwidth}
\centering
\begin{tabular}{|c|c|c|c|}
\hline&&&\\[-0.4cm]
Type& $G$ & $\bar G$&$\pi(G)$\\
\hline &&&\\[-0.4cm]

II&$C_2^3(1)$&$C_2^3$&1\\\hline&&&\\[-0.4cm]

II&$\fD_4(4)$&$C_2^2$&$C_2$\\\hline&&&\\[-0.4cm]

II&$C_2\times C_4(1)$&$C_2^2$&$C_2$\\\hline&&&\\[-0.4cm]

II&$C_2^2\rtimes C_4(1)$&$C_2^3$&$C_2$\\\hline &&&\\[-0.4cm]

II&$C_6(1)$&$1$&$C_6$\\\hline         &&&\\[-0.4cm]

II&$\fS_3(1)$&$1$&$\fS_3$\\\hline&&&\\[-0.4cm]

II&$\fA_4\rtimes C_4(1)$&$C_2^3$&$\fS_3$\\\hline
\end{tabular}
\end{minipage}

\vspace{0.45cm}

\begin{minipage}[t]{0.48\textwidth}
\centering
\begin{tabular}{|c|c|c|c|}
\hline&&&\\[-0.4cm]
Type& $G$ & $\bar G$&$\pi(G)$\\
\hline &&&\\[-0.4cm]

III&$C_2^3(3)$&$C_2^3$&1\\\hline&&&\\[-0.4cm]

III&$\fD_4(6)$&$C_2^2$&$C_2$\\\hline&&&\\[-0.4cm]

III&$C_2\times \fA_4(1)$&$C_2^3$&$C_3$\\\hline&&&\\[-0.4cm]

III&$\fS_3(2)$&$1$&$\fS_3$\\\hline
\end{tabular}
\end{minipage}
\hfill
\begin{minipage}[t]{0.48\textwidth}
\centering
\begin{tabular}{|c|c|c|c|}
\hline&&&\\[-0.4cm]
Type& $G$ & $\bar G$&$\pi(G)$\\
\hline &&&\\[-0.4cm]

IV&$C_2^3(4)$&$C_2^3$&1\\\hline&&&\\[-0.4cm]

IV&$C_{10}(1)$&$C_2$&$C_5$\\\hline
\end{tabular}
\end{minipage}

\vspace{0.45cm}

\begin{minipage}[t]{0.48\textwidth}
\centering
\begin{tabular}{|c|c|c|c|}
\hline&&&\\[-0.4cm]
Type& $G$ & $\bar G$&$\pi(G)$\\
\hline &&&\\[-0.4cm]

V&$C_2^3(4)$&$C_2^3$&1\\\hline&&&\\[-0.4cm]

V&$\fD_4(6)$&$C_2^2$&$C_2$\\\hline &&&\\[-0.4cm]

V&$C_2\times C_4(1)$&$C_2^2$&$C_2$\\\hline&&&\\[-0.4cm]

V&$C_2^2\rtimes C_4(2)$&$C_2^3$&$C_2$\\\hline
\end{tabular}
\end{minipage}
\hfill
\begin{minipage}[t]{0.48\textwidth}
\centering
\begin{tabular}{|c|c|c|c|}
\hline&&&\\[-0.4cm]
Type& $G$ & $\bar G$&$\pi(G)$\\
\hline &&&\\[-0.4cm]

VI&$C_2^3(10)$&$C_2^3$&1\\\hline&&&\\[-0.4cm]

VI&$\fD_4(6)$&$C_2^2$&$C_2$\\\hline
\end{tabular}
\end{minipage}

\vskip 0.2cm
\caption{}
\label{table:list}
\end{table}

With {\tt magma}, we find that every subgroup of $\Aut(X)$ satisfying Condition {\bf (A)} has a fixed point on $X$. We may assume that $X$ does not contain a $G$-invariant line, since otherwise, the $G$-action is 
linearizable, by \cite[Theorem 24]{HT-2quad}.  
Projecting from the $G$-fixed point, we find that $X$ is $G$-birational to a cubic threefold, with singularities of type
$$
4\mathsf A_1, \quad \text{or}\quad 2\mathsf A_3,
$$
containing a unique plane, which must be $G$-invariant.  By results of Section~\ref{sect:sing}, all such cubic threefolds are $G$-unirational.
This proves Theorem~\ref{thm:2-2}.

It would be interesting to classify actions of finite groups on singular intersections of two quadrics $X_{2,2}\subset \bP^5$ and solve the linearization problems for this class of actions. Indeed, these naturally appear in related linearization problems:

\begin{exam}
\label{exam:V5}    
Let $V_5$ be the unique smooth Fano threefold of index $2$ and degree $5$ (the quintic del Pezzo threefold), and 
$$
G\subset \mathrm{Aut}(V_5)\simeq\mathsf{PGL}_2
$$
a finite subgroup, not isomorphic to $\fA_5$. Then the $G$-action is linearizable. Indeed, for cyclic or dihedral groups, linearization follows from the construction in \cite[Section~5.8]{C-Calabi}. If $G=\mathfrak{A}_4$, this follows from the proof of \cite[Lemma 7.8.2]{CS}. If $G=\mathfrak{S}_4$, we have the following $G$-equivariant commutative diagram:  
$$
\xymatrix{
&\tilde{V}_5\ar[ld]_{\alpha}\ar[rd]^{\beta}\ar@{-->}[rr]&&\tilde{X}_{2,2}\ar[ld]_{\gamma}\ar[rd]^{\delta}\\
V_5&&X_{2,2}&&\mathbb{P}^2}
$$
where
\begin{itemize} 
\item 
$\alpha$ is the blowup of the unique $G$-fixed point, 
\item $\beta$ is the contraction of the strict transform of $3$ lines in $V_5$ that pass through the $G$-fixed point,
\item $X_{2,2}\subset \mathbb{P}^5$ is a complete intersection of two quadrics that has three ordinary double points, 
\item $\gamma$ is a small resolution of $X_{2,2}$,
\item 
the dashed arrow is a composition of three Atiyah flops, and 
\item 
$\delta$ is a $\mathbb{P}^1$-bundle, with generically free $G$-action on the base $\mathbb{P}^2$. 
\end{itemize}
This commutative diagram appears in \cite[Theorem~3.2]{langer} and \cite[Theorem~3.6]{Jahnke}. Linearization of the $G$-action follows by applying the no-name lemma. 
\end{exam}

\bibliographystyle{plain}
\bibliography{bguni}

\end{document}